\documentclass{article}

\usepackage{amsmath}
\usepackage{amssymb}
\usepackage{amsfonts}
\usepackage{amsthm}
\usepackage[left=1in,right=1in,top=1in,bottom=1in]{geometry}
\usepackage[width=12cm,format=plain,labelfont=sf,bf,small,textfont=sf,up,small,justification=justified,labelsep=period]{caption}
\usepackage{graphicx}
\usepackage{color}
\usepackage{tikz}
\usetikzlibrary{calc}
\usepackage[backref=page]{hyperref}

\hypersetup{
    colorlinks=true,       % false: boxed links; true: colored links
    linkcolor=blue,          % color of internal links
    citecolor=blue,        % color of links to bibliography
    filecolor=blue,      % color of file links
    urlcolor=blue           % color of external links
}

\makeatletter
\def\iddots{\mathinner{\mkern1mu\raise\p@
\vbox{\kern7\p@\hbox{.}}\mkern2mu
\raise4\p@\hbox{.}\mkern2mu\raise7\p@\hbox{.}\mkern1mu}}
\makeatother

\newcommand{\cC}{\mathcal C}

\newcommand{\cS}{\mathcal S}

\newcommand{\cK}{\mathcal K}
\newcommand{\cF}{\mathcal F}

\newcommand{\cL}{\mathcal L}
\newcommand{\cM}{\mathcal M}
\newcommand{\cT}{\mathcal T}
\newcommand{\cR}{\mathcal R}

\newcommand{\NN}{\mathbb{N}}
\newcommand{\RR}{\mathbb{R}}
\newcommand{\CC}{\mathbb{C}}

\newcommand{\ZZ}{\mathbb{Z}}

\newcommand{\EE}{\mathbb{E}}

\DeclareMathOperator{\supp}{supp}

\DeclareMathOperator{\conv}{conv}
\DeclareMathOperator{\diag}{diag}

\DeclareMathOperator{\xcLP}{xc_{\text{LP}}} % LP extension complexity
\DeclareMathOperator{\xcPSD}{xc_{\text{PSD}}} % PSD extension complexity

\let\Im\relax % Set equal to \relax so that LaTeX thinks it's not defined
\DeclareMathOperator{\Im}{Im}
\let\Re\relax % Set equal to \relax so that LaTeX thinks it's not defined
\DeclareMathOperator{\Re}{Re}

\DeclareMathOperator{\Cay}{Cay} % Cayley graph
\renewcommand{\imath}{\boldsymbol{i}}

\newcommand{\CUT}{\text{CUT}}
\renewcommand{\S}{\mathbf{S}} % Symmetric matrices
\newcommand{\HH}{\mathbf{H}} % Hermitian matrices
\newcommand{\psd}{\succeq}

\renewcommand{\bar}[1]{\overline{#1}}

\newcommand{\Gr}{\mathcal G} % general notation for a graph

{\begin{list}{}%
         {\setlength{\leftmargin}{1cm}\setlength{\rightmargin}{1cm}}%
         \item[]\small%
}
{\end{list}}

\renewcommand{\tilde}[1]{\widetilde{#1}}
\renewcommand{\hat}[1]{\widehat{#1}}

\newcounter{dualtheorem}
\makeatletter
\newtheorem*{dual@theorem}{\dual@title}
\newcommand{\newdualtheorem}[2]{%
\newenvironment{dual#1}[1]{%
 \def\dual@title{#2 \ref*{##1}D}%
 \renewcommand{\thedualtheorem}{\ref*{##1}D}
 \begin{dual@theorem}
 \refstepcounter{dualtheorem}}%
 {\end{dual@theorem}}}
\makeatother

\newcounter{reptheorem}
\makeatletter
\newtheorem*{rep@theorem}{\rep@title}
\newcommand{\newreptheorem}[2]{%
\newenvironment{rep#1}[1]{%
 \def\rep@title{#2 \ref*{##1}}%
 \renewcommand{\thereptheorem}{\ref*{##1}}
 \begin{rep@theorem}
 \refstepcounter{reptheorem}}%
 {\end{rep@theorem}}}
\makeatother

\theoremstyle{plain}
\newtheorem{prop}{Proposition}

\newtheorem{lem}{Lemma}
\newtheorem{thm}{Theorem}

\newtheorem{cor}{Corollary}

\newdualtheorem{thm}{Theorem}
\newdualtheorem{prop}{Proposition}
\newdualtheorem{cor}{Corollary}
\newdualtheorem{lem}{Lemma}

\newreptheorem{thm}{Theorem}
\newreptheorem{prop}{Proposition}
\newreptheorem{cor}{Corollary}
\newreptheorem{lem}{Lemma}

\theoremstyle{definition}
\newtheorem{defn}{Definition}
\newtheorem{problem}{Problem}

\theoremstyle{remark}
\newtheorem{exampleold}{Example}
\newtheorem*{example2}{Example}
\newtheorem*{rem}{Remark}

\newenvironment{example}{\begin{exampleold}}{\hfill $\lozenge$\end{exampleold}}

\newcommand{\ceil}[1]	{\left\lceil #1 \right\rceil}

\newcommand{\triangulationpoweroftwo}[3]{%
\tikzstyle{every node}=[circle, draw, fill=black,
                        inner sep=0pt, minimum width=4pt]
\begin{tikzpicture}[scale=#3]
\def\N{#1}
\def\lmax{#2}

\pgfmathparse{\N-1}
\edef\NN{\pgfmathresult}
\draw (0:1) \foreach \i in {0,1,...,\NN} {
            -- (\i*360/\N:1) node[label={\i*360/\N:\i}]{}
        } -- cycle (0:1);
\foreach \l in {1,...,\lmax} {
     \pgfmathparse{2^(\l)}
     \edef\tl{\pgfmathresult} % increment of our loop below
     \pgfmathparse{2^(\l)*floor(\N/2^(\l))}
     \edef\maxx{\pgfmathresult} % maximum value
    \draw (0:1)
    \foreach \xx in {0,\tl,...,\maxx} {
       -- (\xx*360/\N:1)
    } -- cycle (0:1);
}
\end{tikzpicture}%
}
\newcommand{\triangulationpoweroftwoCyclePowerTwo}[3]{%
\tikzstyle{every node}=[circle, draw, fill=black,
                        inner sep=0pt, minimum width=4pt]
\begin{tikzpicture}[scale=#3]
\def\N{#1}
\def\lmax{#2}

\pgfmathparse{\N-1}
\edef\NN{\pgfmathresult}

\pgfmathparse{\N/2}
\edef\Nd{\pgfmathresult} % \Nd = N/2

\edef\pp{0.175}

\foreach \i in {1,...,\Nd} {
  \pgfmathtruncatemacro\ai{2*\i-2}
  \pgfmathtruncatemacro\bi{2*\i-1}
  \draw (\i*360/\Nd-360/\Nd-\pp*360/\Nd:1) node[label={\i*360/\Nd-360/\Nd-\pp*360/\Nd:\ai},name=x\ai]{};
  \draw (\i*360/\Nd-360/\Nd+\pp*360/\Nd:1) node[label={\i*360/\Nd-360/\Nd+\pp*360/\Nd:\bi},name=x\bi]{};
  \draw (x\ai)--(x\bi);
}

\foreach \l in {0,...,\lmax} {
     \pgfmathparse{2^(\l)}
     \edef\tl{\pgfmathresult} % increment of our loop below
     \pgfmathparse{2^(\l)*(ceil(\Nd/2^(\l))-1)}
     \edef\maxy{\pgfmathresult} % maximum value
    \foreach \yy in {0,\tl,...,\maxy} {
       \pgfmathtruncatemacro\ya{2*\yy}
       \pgfmathtruncatemacro\yb{\ya+1}
       \pgfmathtruncatemacro\za{2*Mod(min(\yy+\tl,\Nd),\Nd)}
       \pgfmathtruncatemacro\zb{\za+1}
       \draw (x\ya)--(x\za);
       \draw (x\ya)--(x\zb);
       \draw (x\yb)--(x\za);
       \draw (x\yb)--(x\zb);
    }
}

\end{tikzpicture}%
}
\newcommand{\triangulationpowerofthree}[3]{%
\tikzstyle{every node}=[circle, draw, fill=black,
                        inner sep=0pt, minimum width=3pt]
\begin{tikzpicture}[scale=#3]
\def\N{#1}
\def\lmax{#2}

\pgfmathparse{\N-1}
\edef\NN{\pgfmathresult}

\draw (0:1) \foreach \i in {0,1,...,\NN} {
            -- (\i*360/\N:1) node[label={\i*360/\N:\i}]{}
        } -- cycle (0:1);

\foreach \l in {0,...,\lmax} {
  \foreach \k in {1,2} {
	\pgfmathparse{floor(\N-\k*3^\l)/3^(\l+1)}
    \foreach \j in {0,...,\pgfmathresult} {
      \pgfmathparse{(3^(\l+1))*\j + (\k)*(3^\l)}
      \edef\xx{\pgfmathresult}
      \pgfmathparse{3^(\l+1)*\j}
      \edef\yy{\pgfmathresult}
      \pgfmathparse{min(\N,(3^(\l+1))*\j + (\k+1)*3^\l)}
      \edef\zz{\pgfmathresult}
      \draw (\xx*360/\N:1) -- (\yy*360/\N:1)--(\zz*360/\N:1)--(\xx*360/\N:1);
    }
  }	
}%
\end{tikzpicture}%
}
\title{Sparse sum-of-squares certificates on finite abelian groups}

\author{Hamza Fawzi \and James Saunderson \and Pablo A. Parrilo\thanks{The authors are with
    the Laboratory for Information and Decision Systems, Department of
    Electrical Engineering and Computer Science, Massachusetts
    Institute of Technology, Cambridge, MA 02139. Email:
    \texttt{\{hfawzi,jamess,parrilo\}@mit.edu}.}}
\renewcommand\footnotemark{}

\date{March 3, 2015}

\begin{document}

\maketitle
% !TEX root = chordal_completion_paper.tex

\begin{abstract}
%\cbstart
Let $G$ be a finite abelian group.
This paper is concerned with nonnegative functions on $G$ that are \emph{sparse} with respect to the Fourier basis.
We establish combinatorial conditions 
on subsets $\cS$ and $\cT$ of Fourier basis elements under which
nonnegative functions with Fourier support $\cS$ are 
sums of squares of functions with Fourier support $\cT$. 
Our combinatorial condition involves constructing a chordal 
cover of a graph related to $G$ and $\cS$ 
(the Cayley graph $\Cay(\hat{G},\cS)$) with maximal cliques 
related to $\cT$.
Our result relies on two main ingredients: the decomposition of sparse positive semidefinite matrices with a chordal sparsity pattern, as well as a simple but key observation exploiting the structure of the Fourier basis elements of $G$ (the characters of $G$).

We apply our general result to two examples. First, in the case where
$G = \ZZ_2^n$, by constructing a particular chordal cover of the 
half-cube graph, we prove that any nonnegative quadratic form in 
$n$ binary variables is a sum of squares of functions of degree at 
most $\ceil{n/2}$, establishing a conjecture of Laurent. Second,
we consider nonnegative functions of degree $d$ on $\ZZ_N$ (when $d$ divides $N$). 
By constructing a particular chordal cover of the $d$th 
power of the $N$-cycle, we prove that any such function 
is a sum of squares of functions with at most $3d\log(N/d)$ 
nonzero Fourier coefficients. Dually this shows that 
a certain cyclic polytope in $\RR^{2d}$ with $N$ vertices
can be expressed as a projection of a section of the cone of 
positive semidefinite matrices of size $3d\log(N/d)$. 
Putting $N=d^2$ gives a family of polytopes in $\RR^{2d}$ with 
linear programming extension complexity $\Omega(d^2)$ and 
semidefinite programming extension complexity $O(d\log(d))$. To the best of our knowledge, this is the first explicit family of polytopes $(P_d)$ in increasing dimensions where $\xcPSD(P_d) = o(\xcLP(P_d))$ (where $\xcPSD$ and $\xcLP$ are respectively the SDP and LP extension complexity).
%\cbend
\end{abstract}

\newpage
\tableofcontents

\section{Introduction}
\label{sec:intro}
% !TEX root = chordal_completion_paper.tex

Let $G$ be a finite abelian group. It is well-known that any function $f:G\rightarrow \CC$ admits a Fourier decomposition where the Fourier basis consists of the \emph{characters} of $G$. 
%For example when $G = \ZZ_N$ we have the well-known discrete Fourier transform $f(x) = \sum_{k \in \ZZ_N} \hat{f}(k) \chi_k(x)$ where the characters are $\chi_k(x) = e^{2i\pi kx/N}$ and $\{\hat{f}(k):k\in \ZZ_N\}$ are the Fourier coefficients of $f$. For the group $G = \ZZ_2^n$ (which we regard as the multiplicative group $\{-1,1\}^n$) the characters are the square-free monomials $\chi_S(x) = \prod_{i \in S} x_i$, for $S\subseteq [n]$, and the Fourier decomposition takes the form $f(x) = \sum_{S \subseteq [n]} \hat{f}(S) \chi_S(x)$.
Such a decomposition takes the form
 \[ f(x) = \sum_{\chi \in \hat{G}} \hat{f}(\chi) \chi(x) \quad \forall x \in G \]
where $\hat{G}$ is the set of characters of $G$ (known as the \emph{dual group} of $G$) and $\hat{f}(\chi)$ are the Fourier coefficients of $f$. The function $f:G\rightarrow \CC$ is called \emph{sparse} if only a few of its Fourier coefficients are nonzero. More precisely we say that $f$ is supported on $\cS \subseteq \hat{G}$ if $\hat{f}(\chi) = 0$ whenever $\chi \notin \cS$.

This paper is concerned with functions $f:G\rightarrow \CC$ that are \emph{sparse} and \emph{nonnegative}, i.e., $f(x) \in \RR_+$ for all $x \in G$. If $f$ is a nonnegative function on $G$, a \emph{sum-of-squares certificate} for the nonnegativity of $f$ has the form:
\begin{equation}
\label{eq:intro:Gsos}
f(x) = \sum_{j=1}^J |f_j(x)|^2 \quad \forall x \in G
\end{equation}
where $f_j:G\rightarrow \CC$. Sum-of-squares certificates of nonnegative functions play an important role in optimization and particularly in semidefinite programming \cite{frgbook}.
%Note that the problems of optimizing a function on $G$, and that of certifying nonnegativity are closely related since we always have:
%\[
%\min_{x \in G} f(x) = \max \{ \gamma : f - \gamma \text{ is nonnegative on $G$}. \}
%\]
%Thus if one has a good characterization of nonnegative functions on $G$ one can solve optimization problems.
When the function $f$ is sparse, it is natural to ask whether $f$ admits a sum-of-squares certificate that is also sparse, i.e., where all the functions $f_j$ are supported on a common ``small'' set $\cT \subseteq \hat{G}$. This is the main question of interest in this paper:
\begin{equation}
\tag{Q}
\label{eq:questionQ}
\parbox[t]{0.8\textwidth}{
Given $\cS \subseteq \hat{G}$, find a subset $\cT \subseteq \hat{G}$ such that any nonnegative function $G\rightarrow \RR_+$ supported on $\cS$ admits a sum-of-squares certificate supported on $\cT$.
}
\end{equation}
Our main result is to give a sufficient condition for a set $\cT$ to satisfy the requirement above for a given $\cS$.
The condition is expressed in terms of \emph{chordal covers} of the Cayley graph $\Cay(\hat{G},\cS)$. Recall that the Cayley graph $\Cay(\hat{G},\cS)$ is the graph where nodes correspond to elements of $\hat{G}$ and where $\chi,\chi'$ are connected by an edge if $\chi^{-1} \chi' \in \cS$.
Our main result can be stated as follows:
\begin{thm}
\label{thm:intro-main-sos}
Let $\cS \subseteq \hat{G}$. Let $\cT$ be a subset of $\hat{G}$ obtained as follows: Let $\Gamma$ be a chordal cover of $\Cay(\hat{G},\cS)$, and for each maximal clique $\cC$ of $\Gamma$, let $\chi_{\cC}$ be an element of $\hat{G}$; define
\begin{equation}
 \label{eq:TGammachiC}
 \cT(\Gamma,\{\chi_{\cC}\}) = \bigcup_{\cC} \chi_{\cC} \cC
\end{equation}
where the union is over all the maximal cliques of $\Gamma$ and where $\chi_{\cC} \cC:=\{\chi_{\cC} \chi : \chi \in \cC\}$ is the translation of $\cC$ by $\chi_{\cC}$. Then any nonnegative function supported on $\cS$ admits a sum-of-squares certificate supported on $\cT(\Gamma,\{\chi_{\cC}\})$.
%Let $\cS \subseteq \hat{G}$. Assume that $\cT$ is a subset of $\hat{G}$ such that the following is true: there exists a chordal cover $\Gamma$ of $\Cay(\hat{G},\cS)$ such that for any maximal clique $\cC$ of $\Gamma$, there exists $\chi_{\cC} \in \hat{G}$ such that $\chi_{\cC} \cC \subseteq \cT$ (where $\chi_{\cC} \cC:=\{\chi_{\cC} \chi : \chi \in \cC\}$ is the translation of $\cC$ by $\chi_{\cC}$). Then any nonnegative function supported on $\cS$ admits a sum-of-squares certificate supported on $\cT$.
\end{thm}
%\cbstart
Theorem \ref{thm:intro-main-sos} gives a way to construct a set $\cT$ that satisfies the condition in \eqref{eq:questionQ} for a given $\cS\subseteq \hat{G}$. Such a construction proceeds in two steps: first choose a chordal cover $\Gamma$ of the graph $\Cay(\hat{G},\cS)$, and then choose elements $\chi_{\cC} \in \hat{G}$ for each maximal clique $\cC$ of $\hat{G}$. Different choices of $\Gamma$ and $\{\chi_{\cC}\}$ will in general lead to different sets $\cT(\Gamma,\{\chi_{\cC}\})$. When using Theorem \ref{thm:intro-main-sos}, one wants to find a good choice of $\Gamma$ and $\{\chi_{\cC}\}$ such that the resulting set $\cT(\Gamma,\{\chi_{\cC}\})$ is as small as possible (or has other desirable properties).

One of the main strengths of Theorem \ref{thm:intro-main-sos} is in the ability to choose the elements $\{\chi_{\cC}\}$. In fact the conclusion of Theorem \ref{thm:intro-main-sos} is almost trivial if $\chi_{\cC} = 1_{\hat{G}}$ for all $\cC$, since in this case it simply says that any nonnegative function has a sum-of-squares certificate supported on $\hat{G}$, which is easy to see since $G$ is finite. As we will see in the applications, it is the ability to \emph{translate} the cliques $\cC$ of $\Gamma$ via the choice of $\chi_{\cC}$ that is key in Theorem \ref{thm:intro-main-sos} and allows us to obtain interesting results.
%The choice of $\{\chi_{\cC}\}$ allow us to translate each clique $\cC$ individually, and
Equation \eqref{eq:TGammachiC} gives us the intuition behind a good choice of $\{\chi_{\cC}\}$: in order to minimize the cardinality of $\cT(\Gamma,\{\chi_{\cC}\})$ one would like to find the translations $\chi_{\cC}$ that maximize the total overlap of the cliques (i.e., minimize the cardinality of their union).

Before describing the main idea behind Theorem \ref{thm:intro-main-sos} and its proof, we illustrate how one can use Theorem \ref{thm:intro-main-sos} in two important special cases, namely $G = \ZZ_2^n$ (the boolean hypercube) and $G=\ZZ_N$.

\begin{itemize}
\item \textbf{Boolean hypercube}: Consider the case $G = \{-1,1\}^n\cong \ZZ_2^n$. The Fourier expansion of functions on $\{-1,1\}^n$ take the form 
\begin{equation}
\label{eq:intro-cube-Fourier}
f(x) = \sum_{S \subseteq [n]} \hat{f}(S) \prod_{i \in S} x_i.
\end{equation}
A function $f$ is said to have degree $d$ if $\hat{f}(S) = 0$ for all $S$ such that $|S| > d$. Many combinatorial optimization problems correspond to optimizing a certain function $f$ over $\{-1,1\}^n$. For example the maximum cut problem in graph theory consists in optimizing a \emph{quadratic function} over $\{-1,1\}^n$. In \cite{laurent2003lower} Laurent conjectured that any nonnegative quadratic function on the hypercube is a sum of squares of functions of degree at most $\lceil n/2 \rceil$. Using our notations, this corresponds to asking whether for $\cS = \{S \subseteq [n] : |S| = 0 \text{ or } 2\}$ one can find $\cT \subseteq \{S \subseteq [n] : |S| \leq \lceil n/2 \rceil\}$ such that the conclusion of Theorem \ref{thm:intro-main-sos} holds. By studying chordal covers of the Cayley graph $\Cay(\hat{G},\cS)$ we are able to answer this question positively:
\begin{thm}
\label{thm:laurent}
Any nonnegative quadratic function on $\{-1,1\}^n$ is a sum-of-squares of polynomials of degree at most $\lceil n/2 \rceil$.
\end{thm}
Note that Blekherman et al. \cite{blekherman2014sums} previously showed a weaker version of the conjecture that allows for multipliers: They showed that for any nonnegative quadratic function $f$ on the hypercube, there exists $h$ sum-of-squares such that $h(x)f(x)$ is a sum-of-squares of polynomials of degree at most $\lceil n/2 \rceil$.

%Also note that another way of interpreting Theorem \ref{thm:laurent} is that the $\lceil n/2 \rceil$ level of the Lasserre hierarchy for the cut polytope is exact. This bound is tight since Laurent showed in \cite{laurent2003lower} that at least $\lceil n/2 \rceil$ levels are needed.

\item \textbf{Trigonometric polynomials}: Another important application that we consider in this paper is the case where $G = \ZZ_N$, the (additive) group of integers modulo $N$. The Fourier decomposition of a function $f:\ZZ_N\rightarrow \CC$ is the usual discrete Fourier transform and takes the form:
\begin{equation}
\label{eq:intro-ZN-Fourier}
f(x) = \sum_{k \in \ZZ_N} \hat{f}(k) e^{2i\pi kx/N}
\end{equation}
where $\hat{f}(k)$ are the Fourier coefficients of $f$. A function $f$ is said to have degree $d$ if $\supp f \subseteq \{-d,-(d-1),\dots,d-1,d\}$. Nonnegative trigonometric polynomials play an important role in many areas such as in signal processing \cite{dumitrescu2007positive}, but also in convex geometry \cite{ziegler1995lectures,alexander2002course}, in their relation to (trigonometric) cyclic polytopes. 
We are interested in nonnegative functions on $G$ of degree at most $d$, i.e., functions supported on $\cS = \{-d,-(d-1),\dots,d-1,d\}$. By studying chordal covers of $\Cay(\hat{G},\cS)$ (which is nothing but the $d$'th power of the cycle graph) and using Theorem \ref{thm:intro-main-sos} we are able to show the following:
%Using Theorem \ref{thm:intro-main-sos} specialized to the case $G = \ZZ_N$ and $\cS = \{-d,-(d-1),\dots,d-1,d\}$ we are able to construct a set $\cT$ where $|\cT| \leq 3 d \log (N/d)$ (assuming $d$ divides $N$) that satisfies the assumption of Theorem \ref{thm:}. In this case the graph $\Cay(\hat{G},\cS)$ is nothing but the $d$'th power of the cycle graph. This leads to the following result:
\begin{thm}
\label{thm:intro-ZN-degd}
%\marginparsmall{Can we say something when $d$ does not divide $N$?}
Let $N$ and $d$ be two integers and assume that $d$ divides $N$. Then there exists $\cT \subseteq \ZZ_N$ with $|\cT| \leq 3 d \log (N/d)$ such that any nonnegative function on $\ZZ_N$ of degree at most $d$ has a sum-of-squares certificate supported on $\cT$.
\end{thm}
\begin{rem}
%\cbstart
Note that if one is interested in functions of degree at most $d$ on $\ZZ_N$ and $d$ does not divide $N$, then one can still apply Theorem \ref{thm:intro-ZN-degd} with $d'$ instead of $d$, where $d'$ is the smallest divisor of $N$ that is greater than $d$.
%\cbend
% The set $\cT$ we thus get will clearly satisfy the desired property, namely that any nonnegative function on $\ZZ_N$ of degree at most $d$ has a sum-of-squares certificate supported on $\cT$.
\end{rem}
\end{itemize}

\paragraph{Dual point of view and moment polytopes} Theorem \ref{thm:intro-main-sos} can be interpreted from the dual point of view as giving a semidefinite programming description of certain moment polytopes. If $\cS \subseteq \hat{G}$, define the \emph{moment polytope} $\cM(G,\cS)$ to be the set of $\cS$-moments of probability distributions on $G$, i.e.,
\[
\cM(G,\cS) = \Bigl\{ \left(\EE_{x\sim \mu}\bigl[\chi(x)\bigr]\right)_{\chi\in \cS}\in \CC^{\cS}: \textup{$\mu$ a probability measure supported on $G$}\Bigr\}.
\]
Note that $\cM(G,\cS)$ is a polytope since it can be equivalently expressed as:
\[ \cM(G,\cS) = \conv \Bigl\{(\chi(x))_{\chi\in \cS}\in \CC^{\cS}: x\in G\Bigr\}. \]
Note that from a geometric point of view, nonnegative functions $f:G\rightarrow \RR_+$ supported on $\cS$ correspond to valid linear inequalities for the polytope $\cM(G,\cS)$. By giving a sum-of-squares characterization for all valid inequalities of $\cM(G,\cS)$ Theorem \ref{thm:intro-main-sos} allows us to obtain a semidefinite programming description of $\cM(G,\cS)$. The following statement can be obtained from Theorem \ref{thm:intro-main-sos} by duality (we call this result ``Theorem 1D'' to reflect that it is a dual version of ``Theorem 1''--we adopt this numbering convention throughout the paper):
\begin{dualthm}{thm:intro-main-sos}
\label{thm:intro-main-moments}
Let $\cS \subseteq \hat{G}$ and let $\cT = \cT(\Gamma,\{\chi_{\cC}\})$ be as defined in Theorem \ref{thm:intro-main-sos}. Then we have the following semidefinite programming description of the moment polytope $\cM(G,\cS)$:
%Under the hypotheses of Theorem \ref{thm:intro-main-sos} 
\begin{equation}
\label{eq:intro-Mpsdlift}
\begin{aligned}
\cM(G,\cS) 
%&= \Bigl\{ \bigl(\ell(\chi)\bigr)_{\chi \in \cS} \text{ where } \ell \text{ s.t. } \ell(1_{\hat{G}}) = 1 \text{ and } \ell\bigl(|f|^2\bigr) \geq 0 \text{ for all } f:G\rightarrow \CC \text{ supported on } \cT \Bigr\}\\
&= \begin{aligned}[t]\Bigl\{(\ell_\chi)_{\chi\in \cS} \;\; : \;\; \exists (y_{\chi})_{\chi\in \cT^{-1}\cT}\;\;\text{such that}\;\;
            & y_\chi = \ell_\chi \text{ for all $\chi\in \cS$ }, \text{ and }\\
    &y_{1_{\hat{G}}}=1, \text{ and } \bigl[y_{\bar{\chi}\chi'}\bigr]_{\chi,\chi'\in \cT} \psd 0\Bigr\}.
\end{aligned}
\end{aligned}
\end{equation}
\end{dualthm}
In terms of positive semidefinite lifts, Equation \eqref{eq:intro-Mpsdlift} shows that $\cM(G,\cS)$ has a \emph{Hermitian positive semidefinite lift} of size $|\cT|$. We now illustrate this dual point of view for the two applications mentioned above, $G = \{-1,1\}^n$ and $G = \ZZ_N$:

\begin{itemize}
\item For the case of the boolean hypercube $G =\{-1,1\}^n$, if $\cS = \{S \subseteq [n] : |S| = 0 \text{ or } 2\}$, the moment polytope $\cM(\{-1,1\}^n,\cS\setminus\{\emptyset\})$ is nothing but the \emph{cut polytope} for the complete graph on $n$ vertices which we denote by $\CUT_n$:
\[
\CUT_n = \conv \Bigl\{ (x_i x_j)_{i < j} \in \RR^{\binom{n}{2}} : x \in \{-1,1\}^n \Bigr\}.
\]
From the dual point of view, Theorem \ref{thm:laurent} shows that the $\lceil n/2 \rceil$ level of the Lasserre hierarchy for the cut polytope is exact. This bound is tight since Laurent showed in \cite{laurent2003lower} that at least $\lceil n/2 \rceil$ levels are needed.

\begin{dualthm}{thm:laurent}
\label{thm:intro-laurent-dual}
The $\lceil n/2 \rceil$ level of the Lasserre hierarchy for the cut polytope $\CUT_n$ (as considered in \cite{laurent2003lower}) is exact.
\end{dualthm}

\item Consider now the case $G = \ZZ_N$ and $\cS = \{-d,-(d-1),\dots,d-1,d\}$. Here the moment polytope $\cM(G,\cS)$ is the \emph{trigonometric cyclic polytope} of degree $d$ which we denote by $TC(N,2d)$:
\begin{equation}
\label{eq:intro-TC}
 TC(N,2d) = \conv \Bigl\{ M(2\pi x/N) : x=0,1,\dots,N-1 \Bigr\} \subset \RR^{2d},
\end{equation}
where $M(\theta)$ is the degree $d$ trigonometric moment curve:
\[ M(\theta) = \Bigl( \cos (\theta), \sin (\theta), \cos(2\theta), \sin(2\theta), \dots, \cos(d\theta), \sin(d\theta) \Bigr).
\]
When interpreted from the dual point of view, Theorem \ref{thm:intro-ZN-degd} shows that $TC(N,2d)$ has a Hermitian positive semidefinite lift of size at most $3d \log (N/d)$.
\begin{dualthm}{thm:intro-ZN-degd}
\label{thm:intro-TClift}
Let $N$ and $d$ be two integers and assume that $d$ divides $N$. The trigonometric cyclic polytope $TC(N,2d)$ defined in \eqref{eq:intro-TC} has a Hermitian positive semidefinite lift of size at most $3d\log (N/d)$.
\end{dualthm}
Note that in the case $d=1$ the polytope $TC(N,2d)$ is nothing but the regular $N$-gon in the plane. Theorem \ref{thm:intro-TClift} thus recovers, and extends to the case where $N$ is not a power of two, a result from \cite{polygonsequivariant} giving a semidefinite lift of the regular $N$-gon of size $O(\log N)$.

For $d > 1$ our result is, as far as we are aware, the first nontrivial semidefinite programming lift of a cyclic polytope. Furthermore, in the regime where $N=d^2$ our lift is provably smaller than any linear programming lift: Indeed, since $TC(d^2,2d)$ is $d$-neighborly \cite{gale1963neighborly}, a lower bound from \cite{fiorini2013combinatorial} concerning neighborly polytopes shows that any linear programming lift of $TC(d^2,2d)$ must have size at least $\Omega(d^2)$, whereas our semidefinite programming lift in this case has size $O(d\log d) = o(d^2)$. To the best of our knowledge this gives the first example of a family of polytopes $(P_d)_{d \in \NN}$ in increasing dimensions where $\xcPSD(P_d) = o(\xcLP(P_d))$ where $\xcPSD$ and $\xcLP$ are respectively the SDP and LP extension complexity (see Section \ref{sec:lifts} for the definitions). More precisely, we have:
\begin{cor}
\label{cor:LPSDPgap}
There exists a family $(P_d)_{d \in \NN}$ of polytopes where $P_d \subset \RR^{2d}$ such that 
\[ \frac{\xcPSD(P_d)}{\xcLP(P_d)} = O\left( \frac{\log d}{d} \right). \]
\end{cor}
The only nontrivial linear programming lift for cyclic polytopes that we are aware of is a construction by Bogomolov et al. \cite{bogomolov2014small} for the polytope $\conv \{ (i,i^2,\dots,i^d) : i=1,\dots,N \}$
which has size $(\log N)^{\lfloor d/2 \rfloor}$.
\end{itemize}

\paragraph{Main ideas} We now briefly describe the main ideas behind Theorem \ref{thm:intro-main-sos}, which can be summarized in three steps:
\begin{enumerate}
\item \emph{A sum-of-squares certificate with a sparse Gram matrix}: Given a nonnegative function $f:G\rightarrow \RR_+$ it is easy to see, since $G$ is finite, that $f$ can be written as a sum-of-squares. When the function $f$ is supported on $\cS$, one can show that $f$ admits a specific sum-of-squares representation where the \emph{Gram matrix} $Q$, in the basis of characters, is sparse according to the graph $\Cay(\hat{G},\cS)$.
\item \emph{Chordal completion}: Let $\Gamma$ be a chordal cover of the graph $\Cay(\hat{G},\cS)$. Using well-known results concerning positive semidefinite matrices that are sparse according to a chordal graph \cite{griewank1984existence,grone1984positive}  (see Section \ref{sec:chordal} for more details) one can decompose the Gram matrix $Q$ into a sum of positive-semidefinite matrices, where each matrix is supported on a maximal clique of $\Gamma$. In terms of sum-of-squares representation, this means that the function $f$ can be written as:
\begin{equation}
 \label{eq:intro-sos-decomp1}
 f = \sum_{j} |f_j|^2
\end{equation}
where each $f_j$ is supported on a maximal clique $\cC_j$ of $\Gamma$.
\item \emph{Translation of cliques}: The problem with the decomposition \eqref{eq:intro-sos-decomp1} is that even though each maximal clique $\cC_j$ might be small, the union of the $\cC_j$'s might be large, and thus the total support of \eqref{eq:intro-sos-decomp1} might be large (in fact the union of the $\cC_j$ is the whole $\hat{G}$). In order to reduce the total support of the sum-of-squares certificate \eqref{eq:intro-sos-decomp1}, we use the following simple but crucial observation: if $h$ is a function supported on $\cC$ and if $\chi \in \hat{G}$ then $\chi h$ is supported on $\chi \cC$ and we have $|\chi h|^2 = |h|^2$. Thus if for each maximal clique $\cC_j$ of $\Gamma$ we choose a certain $\chi_{j} \in \hat{G}$ then, by translating each term in \eqref{eq:intro-sos-decomp1} by $\chi_{j}$ we obtain a sum-of-squares representation of $f$ of the form $f = \sum_{j} |\tilde{h}_j|^2$ where $\tilde{h}_{j}$ is supported in $\chi_{j} \cC_j$. Having chosen the $\chi_{j}$ such that $\chi_{j} \cC_j \subseteq \cT$ for all maximal cliques $\cC_j$ (cf.  Theorem \ref{thm:intro-main-sos}), we get a representation of $f$ as a sum-of-squares of functions supported on $\cT$.
\end{enumerate}

\paragraph{Organization} The paper is organized as follows. Section \ref{sec:prelim} starts by giving a brief review of Fourier analysis of finite abelian groups, as well as a review of chordal graphs, chordal covers and the main results concerning decomposition/matrix completion with chordal sparsity structure \cite{griewank1984existence,grone1984positive}.
In Section \ref{sec:general} we prove our main result, Theorem \ref{thm:intro-main-sos}. We present the proof using the two dual viewpoints of sum-of-squares certificates and in terms of moment polytopes.
In Section \ref{sec:maxcut} we look at the case of the hypercube $G=\{-1,1\}^n$ mentioned earlier, and we look in particular at quadratic functions on the hypercube. We give an explicit chordal cover for the corresponding Cayley graph and we show how it leads to a proof of Laurent's conjecture.
In Section \ref{sec:cyclic} we look at the special case $G=\ZZ_N$ and functions of degree $d$. We give an explicit chordal cover for the corresponding graphs, and we discuss the consequences concerning positive semidefinite lifts of the trigonometric cylic polytope.

\paragraph{Notations} We collect some of the notations used in the paper. If $z \in \CC$ we denote by $\bar{z}$ the complex conjugate of $z$. Given a square matrix $X \in \CC^{n\times n}$ the Hermitian conjugate of $X$ is denoted $X^{*}$, and $X$ is called Hermitian if $X^* = X$. The space of $n\times n$ Hermitian matrices is denoted $\HH^n$ and the cone of Hermitian positive semidefinite matrices is denoted by $\HH^n_+$. Similarly we denote by $\S^n$ the space of $n\times n$ real symmetric matrices and by $\S^n_+$ the cone of $n\times n$ real symmetric positive semidefinite matrices. If $V$ is an arbitrary set, we will denote by $\CC^V$ the space of complex vectors indexed by elements of $V$, and by $\HH^V$ the space of Hermitian matrices where rows and columns are indexed by elements of $V$ (and similarly for $\HH^V_+$ and $\S^V, \S^V_+$).

% ------
\section{Preliminaries}
\label{sec:prelim}
% !TEX root = chordal_completion_paper.tex

In this section we present some background material needed for the paper: we first recall some of the basic results and terminology concerning Fourier analysis on finite abelian groups \cite{rudin1990fourier,terras1999fourier}, then we review the definition of chordal graph and the main results concerning sparse positive semidefinite matrices and matrix completion. We also review some of the terminology concerning lifts of polytopes/extended formulations.

\subsection{Fourier analysis on finite groups}

Let $G$ be a finite abelian group which we denote multiplicatively, and let $\cF(G,\CC)$ be the vector space of complex-valued functions on $G$. A \emph{character} $\chi$ of $G$ is a \emph{group homomorphism} $\chi : G\rightarrow (\CC^*,\times)$, i.e., it is an element of $\cF(G,\CC)$ which satisfies:
\[ \chi(xy) = \chi(x) \chi(y) \;\; \forall x,y \in G. \]
Since $G$ is abelian, one can easily show that the (pointwise) product of two characters is a character and that the (pointwise) inverse of a character is again a character. Thus if we denote by $\hat{G}$ the set of characters of $G$, then $\hat{G}$ forms an abelian group, where the group operation corresponds to pointwise multiplication. The group $\hat{G}$ is known as the \emph{dual group} of $G$.
Observe that since $G$ is finite, if $\chi$ is a character then for any $x \in G$ we have $\chi(x)^{|G|} = \chi(x^{|G|}) = \chi(1_G) = 1$, which implies that $|\chi(x)|=1$. It follows that the inverse of a character $\chi$ is simply its (pointwise) complex conjugate $\bar{\chi}$.

Consider the standard inner product on $\cF(G,\CC)$:
\begin{equation}
 \label{eq:ip}
 \langle f, g \rangle = \frac{1}{|G|} \sum_{x \in G} \overline{f(x)} g(x) \quad \forall f, g\in\cF(G,\CC).
\end{equation}
A crucial property of the set of characters $\hat{G}$ is that they form an orthonormal basis of $\cF(G,\CC)$, which is called the \emph{Fourier basis} of $G$. Note that this implies in particular that $|\hat{G}| = |G|$. We summarize this in the following theorem:

\begin{thm}
\label{thm:charactersONB}
Let $G$ be a finite abelian group and let $\hat{G}$ be the set of characters of $G$. Then $\hat{G}$ is an abelian group with pointwise multiplication. Furthermore $|\hat{G}|=|G|$ and $\hat{G}$ forms an orthonormal basis of $\cF(G,\CC)$ for the standard inner product \eqref{eq:ip}.
\end{thm}
We now illustrate the previous theorem in the two examples $G = \{-1,1\}^n$ (the hypercube) and $G=\ZZ_N$ presented in the introduction.

\begin{example}[Fourier analysis on the hypercube]
Let $G = \{-1,1\}^n$ be the hypercube in dimension $n$ which forms a group of size $2^n$ under componentwise multiplication, isomorphic to $\ZZ_2^n$. Observe that if $S$ is a subset of $[n]$ then the function $\chi_S$ defined by:
\[
\chi_S : \{-1,1\}^n \rightarrow \CC^*, \quad \chi_S(x) = \prod_{i \in S} x_i
\]
satisfies $\chi_S(xy) = \chi_S(x) \chi_S(y)$, and thus is a character of $G$. For example $\chi_{\emptyset}$ is the constant function equal to 1, and $\chi_{[n]}$ is the function $\chi_{[n]}(x) = x_1\dots x_n$. One can show that these are all the characters of $G$, i.e., $\hat{G} = \{\chi_S, S \subseteq [n]\}$. Thus the decomposition of a function $f:\{-1,1\}^n\rightarrow \CC$ in the basis of characters takes the form:
\[ f(x) = \sum_{S \subseteq [n]} \hat{f}(S) \prod_{i \in S} x_i, \]
where $\hat{f}(S)$ are the Fourier coefficients of $f$.
\end{example}

\begin{example}[Fourier analysis on $\ZZ_N$]
\label{ex:charactersZN}
Let $N$ be an integer and consider the (additive) group $G = \ZZ_N$ of integers modulo $N$.
For $k \in \ZZ_N$, define $\chi_k$ by
\[ \chi_k:\ZZ_N \rightarrow \CC^*, \quad \chi_k(x) = e^{2i\pi k x / N}. \]
Note that $\chi_k$ satisfies $\chi_k(x+y) = \chi_k(x)\chi_k(y)$ and thus $\chi_k$ is a character of $\ZZ_N$. It is not hard to show that any character $\chi$ of $\ZZ_N$ actually must have the form $\chi = \chi_k$ for some $k \in \ZZ_N$. Thus the dual group $\hat{\ZZ_N}$ of $\ZZ_N$ is $\hat{\ZZ_N} = \{\chi_k, k \in \ZZ_N\}$. Note that $\chi_{k} \chi_{k'} = \chi_{k+k'}$ and $(\chi_k)^{-1} = \bar{\chi_k} = \chi_{-k}$, and thus $\hat{\ZZ_N}$ is isomorphic to $\ZZ_N$.
According to Theorem \ref{thm:charactersONB}, any function $f:\ZZ_N \rightarrow \CC$ can be decomposed in the basis of characters:
\[ f(x) = \sum_{k \in \ZZ_N} \hat{f}(k) e^{2i\pi k x/N} \quad \forall x \in \ZZ_N. \]
This decomposition is nothing but the well-known Fourier decomposition of discrete signals of length $N$.
\end{example}

For a general finite abelian group $G$, the Fourier decomposition of a function $f:G\rightarrow \CC$, in the orthonormal basis of characters takes the form:
\[
f(x) = \sum_{\chi \in \hat{G}} \hat{f}(\chi) \chi(x).
\]
The coefficients $\hat{f}(\chi)$ are the \emph{Fourier coefficients} of $f$. By orthonormality of the basis of characters, we have for any $\chi \in \hat{G}$:
\[
\hat{f}(\chi) = \langle \chi, f \rangle = \frac{1}{|G|} \sum_{x \in G} \overline{\chi(x)} f(x).
\]
The \emph{support} of a function $f$, denoted $\supp f$ is the set of characters $\chi$ for which $\hat{f}(\chi)\neq 0$:
\[ \supp f = \{ \chi \in \hat{G} : \hat{f}(\chi) \neq 0 \}. \]
%Also note that if $f, g \in \cF(G,\CC)$ then the following identity is true, usually known as Plancherel's theorem:
%\[ \frac{1}{|G|} \sum_{x \in G} \overline{f(x)} g(x) = \sum_{\chi \in \hat{G}} \overline{\hat{f}(\chi)} \hat{g}(\chi). \]

% and let $\hat{G}$ be the dual group consisting of group homomorphisms $G \rightarrow \CC^{*}$. The elements of $\hat{G}$ form an orthonormal basis of $\CC^{G}$, the space of complex-valued functions on $G$, with the natural inner product:
%\[ \langle f, g \rangle = \frac{1}{|G|} \sum_{x \in G} f(x) \overline{g(x)}. \]
%
%Given $f \in \CC^{G}$, we denote by $\hat{f}:\hat{G}\rightarrow \CC$ its ``Fourier transform'' defined by:
%\[
%\hat{f}(\chi) = \langle f, \chi \rangle = \frac{1}{|G|} \sum_{x \in G} f(x) \overline{\chi(x)} \quad \forall \chi \in \hat{G}.
%\]
%We have:
%\[
%f(x) = \sum_{\chi \in \hat{G}} \hat{f}(\chi) \chi(x) \quad \forall x \in G.
%\]
%Also Plancherel's formula says that:
%\[ \frac{1}{N} \sum_{x \in G} f(x) \overline{g(x)} = \sum_{\chi \in \hat{G}} \hat{f}(\chi) \overline{\hat{g}(\chi)}. \]
%Another basis of $\CC^G$ is the standard Dirac basis $(\delta_x)_{x \in G}$ given by:
%\[ \delta_x(x') = \begin{cases} 1 & \text{if } x = x'\\ 0 & \text{else} \end{cases} \quad \forall x,x'\in G. \]

\subsection{Chordal graphs and matrix completion}
\label{sec:chordal}

In this section we recall some of the main results concerning sparse matrix decomposition and matrix completion with a chordal sparsity structure. For more details, we refer the reader to \cite{grone1984positive,griewank1984existence} and \cite{agler1988positive}.

\paragraph{Chordal graphs} Let $\Gr = (V,E)$ be a graph.  The graph $\Gr$ is called \emph{chordal} if any cycle of length at least four has a chord. A \emph{chordal cover} (also called \emph{triangulation}) of $\Gr$ is a graph $\Gr'=(V,E')$ where $E \subset E'$ and where $\Gr'$ is chordal. Figure \ref{fig:chordal_graph} shows a non-chordal graph $\Gr$ on four vertices and a chordal cover $\Gr'$ of $\Gr$.
\begin{figure}[ht]
  \centering
  \includegraphics[scale=1]{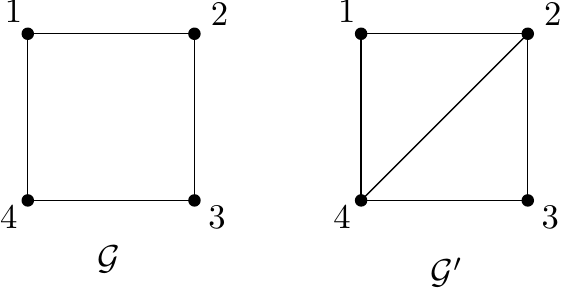}
  \caption{A non-chordal graph $\Gr$ and a chordal cover $\Gr'$ of $\Gr$.}
  \label{fig:chordal_graph}
\end{figure}

A subset $\cC \subseteq V$ is a \emph{clique} in $\Gr$ if $\{i,j\} \in E$ for all $i,j \in \cC$, $i\neq j$. The clique $\cC$ is called \emph{maximal} if it is not a strict subset of another clique $\cC'$ of $\Gr$. For example the maximal cliques of the graph $\Gr'$ shown in Figure \ref{fig:chordal_graph} are $\{1,2,4\}$ and $\{2,3,4\}$. 
%The size of the largest clique in a graph $G$ is known as the \emph{cliqe number} and is denoted by $\omega(G)$.
%An important parameter in the study of chordal covers of graphs is that of \emph{treewidth}. The treewidth of a graph $G$ is the minimum of $\omega(G')-1$ over all chordal covers $G'$ of $G$, i.e.,
%\[ \treewidth(G) = \min_{\text{$G'$ chordal cover of $G$}} \omega(G')-1. \]
%The treewidth plays an important role in the design of graph algorithms and in statistical inference, see e.g., \cite{wainwright2008graphical,chandrasekaran2008complexity}.

\paragraph{Sparse matrices} 
Let $Q \in \HH^V$ be a Hermitian positive semidefinite matrix where rows and columns are indexed by some set $V$. Assume furthermore that $Q$ is sparse according to some graph $\Gr=(V,E)$, i.e.,
\[ Q_{ij} \neq 0, i\neq j \Rightarrow \{i,j\} \in E. \]
%Given a graph $G=(V,E)$ where $V=\{1,\dots,n\}$, a Hermitian matrix $Q \in \HH^n$ is called $G$-sparse if $ij \notin E, i\neq j \Rightarrow Q_{ij} = 0$. Let $\Hsparse{G} \subseteq \HH^n$ be the space of $G$-sparse matrices and note that $\Hsparse{G} \cong \CC^{E \cup D(V)}$ where $D(V)=\{(i,i), i \in V\}$ are the diagonal entries. In this paper we will be interested in sparse positive semidefinite matrices, i.e., matrices that live in the cone $\Hsparse{G} \cap \HH^n_+$. 
One of the main tools used in this paper is a result from \cite{griewank1984existence,grone1984positive} which allows to decompose sparse positive semidefinite matrices as a sum of positive semidefinite matrices supported on a small subset of rows/columns. We say that a Hermitian matrix $A$ is supported on $\cC \subseteq V$ if $A_{ij} = 0$ whenever $i \notin \cC$ or $j \notin \cC$. The result can be stated as follows:
%Sparse positive semidefinite matrices have been studied before in \cite{griewank1984existence,grone1984positive} and one of the main results concerning such matrices is the following decomposition theorem in the case where $G$ is chordal.
\begin{thm} (\cite{griewank1984existence,grone1984positive})
\label{thm:sparsematrixdecomposition}
Let $Q$ be a Hermitian positive semidefinite matrix, and assume that $Q$ is sparse according to some graph $\Gr$. Assume furthermore that $\Gr$ is chordal. Then for every maximal clique $\cC$ of $\Gr$ there exists a Hermitian positive semidefinite matrix $Q_{\cC}$ supported on $\cC$ such that:
\begin{equation}
 \label{eq:Qdecomp}
 Q = \sum_{\cC} Q_{\cC}.
\end{equation}
\end{thm}
\begin{rem}
If the sparsity pattern $\Gr$ of $Q$ is not chordal, one can still apply the previous theorem by considering a chordal cover $\Gr'$ of $\Gr$. Indeed if $Q$ is sparse according to $\Gr$ then it also clearly sparse according to $\Gr'$, since $\Gr\subseteq \Gr'$. In this case the summation \eqref{eq:Qdecomp} is over the maximal cliques of $\Gr'$.
\end{rem}
\begin{example}
We can illustrate the previous theorem with a simple $4\times 4$ matrix. Let $Q$ be the $4\times 4$ Hermitian positive semidefinite matrix given by:
\[
Q = \begin{bmatrix}
2 & 1-i & 0 & 1+i\\
1+i & 2 & 1-i & 0\\
0 & 1+i & 2 & 1-i\\
1-i & 0 & 1+i & 2
\end{bmatrix}.
\]
Note that $Q$ is sparse according to the ``square graph'' $\Gr$ shown in Figure \ref{fig:chordal_graph}(left). Since $\Gr$ is not chordal we cannot directly apply Theorem \ref{thm:sparsematrixdecomposition} with $\Gr$, but we can apply it with $\Gr'$ shown in Figure\ref{fig:chordal_graph}(right) which is a chordal cover of $\Gr$. In this case Theorem \ref{thm:sparsematrixdecomposition} asserts that one can decompose $Q$ as a sum of two positive semidefinite matrices supported respectively on the maximal cliques, $\{1,2,4\}$ and $\{2,3,4\}$. For this example, it is not hard to find an explicit decomposition, for example we can verify that:
\[
Q = 
\underbrace{
\begin{bmatrix}
2   & 1-i & 0 & 1+i\\
1+i & 1   & 0 & i\\
0   & 0   & 0 & 0\\
1-i & -i  & 0 & 1
\end{bmatrix}
}_{\succeq 0}
+
\underbrace{
\begin{bmatrix}
0 & 0 & 0 & 0\\
0 & 1 & 1-i & -i\\
0 & 1+i & 2 & 1-i\\
0 & i & 1+i & 1
\end{bmatrix}
}_{\succeq 0}.
\]
\end{example}

\paragraph{Matrix completion} One can also state Theorem \ref{thm:sparsematrixdecomposition} in its dual form, in terms of the \emph{matrix completion problem}. Given a graph $\Gr=(V,E)$, a \emph{$\Gr$-partial matrix} $X$ is a matrix where only the diagonal entries, as well as the entries $X_{ij}$ for $\{i,j\} \in E$ are specified. Given a $\Gr$-partial matrix $X$, the positive semidefinite matrix completion problem asks whether $X$ can be completed into a full $|V|\times |V|$ Hermitian matrix that is positive semidefinite. 
%Observe that the space of $G$-partial matrices $\Hpartial{G}$ can be seen as the dual of $\Hsparse{G}$, and that the dual of the cone $\Hsparse{G} \cap \HH^n_+$ is precisely the cone of $G$-partial matrices that can be completed into a positive semidefinite matrix.
Clearly a necessary condition for such a completion to exist is that $X[\cC,\cC] \succeq 0$ for all cliques $\cC$ of $\Gr$ (note that if $\cC$ is a clique of $\Gr$, then all the entries of $X[\cC,\cC]$ are specified). When $\Gr$ is chordal, it turns out that this condition is also sufficient. The following theorem can actually be obtained from Theorem \ref{thm:sparsematrixdecomposition} via duality:
\begin{thm} (\cite{grone1984positive})
    \label{thm:completion}
Let $\Gr=(V,E)$ be a graph and let $X$ be a $\Gr$-partial matrix. Assume that $\Gr$ is chordal.
Then $X$ can be completed into a full $|V|\times |V|$ Hermitian positive semidefinite matrix if, and only if, $X[\cC,\cC] \succeq 0$ for all maximal cliques $\cC$ of $\Gr$.
\end{thm}

%Let $G=(V,E)$ be a graph where $V=\{1,\dots,n\}$. Let $\H^n_G$ be the space of \emph{partial $G$-matrices}, i.e., the space of matrices where only the entries corresponding to $E$, as well as the diagonal entries, are specified. Note that $\H^n_G \cong \CC^{E\cup D}$ where $D=\{(i,i), i\in V\}$ correspond to the diagonal entries. Given $X \in \H^n_G$ a $G$-partial matrix, we say that $X$ can be completed into a positive semidefinite matrix if there exist values $X_{ij}$ for $ij \notin E, i\neq j$ such that the complete matrix is positive semidefinite. Clearly, a necessary condition for such a completion to exist is that the principal submatrices where $X$ is completely specified are all positive semidefinite; in other words we must have $X[C,C] \succeq 0$ where $C$ are the (maximal) cliques of $G$. It turns out that when the graph $G$ is chordal, such a condition is also sufficient.

\subsection{Lifts of polytopes}
\label{sec:lifts}

In this section we recall some of the definitions and terminology concerning \emph{lifts (or extended formulations) of polytopes}. The concepts defined here are not used in the proofs of our theorems, but simply make some of the results more convenient to state. We refer the reader to \cite{yannakakis1991expressing,gouveia2011lifts} for more details. 

Let $P \subset \RR^d$ be a polytope. We say that $P$ has a \emph{LP lift} of size $k$ if $P$ can be expressed as the linear projection of an affine section of the cone $\RR^k_+$, i.e., if there exist $\pi:\RR^k \rightarrow \RR^d$ linear and an affine subspace $L \subset \RR^k$ such that:
\begin{equation}
 \label{eq:LPlift}
 P = \pi(\RR^k_+ \cap L).
\end{equation}
Note that this definition is equivalent to say that $P$ is the projection of a polytope $Q$ with $k$ facets. 
%The main motivation to study LP lifts of polytopes is the following: If $P$ has an LP lift of size $k$, then optimizing a linear function over $P$ can be written as a linear program with $k$ inequality constraints. Indeed if \eqref{eq:LPlift} holds and if $\ell \in (\RR^d)^*$ is any arbitrary linear function on $\RR^d$, then:
%\[ \max_{x \in P} \ell(x) = \max_{y \in \RR^k} \ell\circ \pi (y) \text{ subject to } y \in L, y \geq 0. \]
The smallest $k$ such that $P$ has a LP lift of size $k$ is called the \emph{LP extension complexity} of $P$ and is denoted $\xcLP(P)$.

The definition of LP lift can be extended to \emph{PSD lifts}, where instead we are looking to describe $P$ using \emph{linear matrix inequalities}. Formally, we say that $P$ has a \emph{Hermitian PSD lift} of size $k$ if $P$ can be expressed as the linear projection of an affine section of the Hermitian positive semidefinite cone $\HH^k_+$, i.e., if there exist $\pi:\HH^k\rightarrow \RR^d$ linear, and an affine subspace $L \subset \HH^k$ such that:
\begin{equation}
 \label{eq:PSDlift}
 P = \pi(\HH^k_+ \cap L).
\end{equation}
%Observe that if $P$ has a Hermitian PSD lift of size $k$, then optimizing a linear function over $P$ can be written as a Hermitian semidefinite program over the cone $\HH^k_+$:
%\[ \max_{x \in P} \ell(x) = \max_{Y \in \HH^k} \ell \circ \pi(Y) \text{ subject to } Y \in L, Y \succeq 0, \]
%where $Y \succeq 0$ indicates that $Y$ is Hermitian positive semidefinite.
The smallest $k$ for which $P$ has a PSD lift of size $k$ is called the \emph{PSD extension complexity} of $P$ and denoted $\xcPSD(P)$.
Note that one can also define PSD lifts with the cone of real symmetric positive semidefinite matrices $\S^k_+$ (instead of $\HH^k_+$) and in this case we call the lift a \emph{real PSD lift}.

\section{Main result for general finite abelian groups}
\label{sec:general}

In this section we state and prove our main result in the general setting of finite abelian groups $G$. We first describe the primal point of view concerning sparse sum-of-squares certificates of nonnegative functions, and then we present the dual point of view related to moment polytopes.

\subsection{Nonnegative functions and sum-of-squares certificates}
\label{sec:sos}
% !TEX root = chordal_completion_paper.tex

Let $G$ be a finite abelian group and let $\cF(G,\CC)$ be the space of complex-valued functions on $G$. Given a nonnegative function $f:G\rightarrow \RR_+$, a \emph{sum-of-squares certificate} for $f$ takes the form:
\begin{equation}
\label{eq:sosrep}
f(x) = \sum_{k=1}^K |f_k(x)|^2 \quad \forall x \in G.
\end{equation}
where $f_1,\dots,f_K \in \cF(G,\CC)$.
%Sum-of-squares certificates for nonnegative functions have been studied extensively since the works of \cite{parrilo2000structured,nesterov2000squared,lasserre2001global}. An 
%We start by recalling the following well-known proposition from  which gives an important connection between sum-of-squares certificates and positive semidefinite Hermitian matrices. We include a proof for completeness.

It is well-known in the literature on polynomial optimization (see e.g., \cite{parrilo2000structured,nesterov2000squared,lasserre2001global}) that the existence of sum-of-squares certificates can be expressed in terms of the existence of a certain positive semidefinite matrix called a \emph{Gram} matrix for $f$. This connection between sum-of-squares certificates and positive semidefinite matrices will be important in this paper, and so we recall this connection more formally in the next proposition:
% This is the object of the next proposition:
\begin{prop}
\label{prop:sos}
Let $n = |G|$ and let $b_1,\dots,b_n$ be a basis for $\cF(G,\CC)$. Let $f:G\rightarrow \RR$ be a real-valued function on $G$. Then $f$ has a sum-of-squares representation \eqref{eq:sosrep}, if, and only if, there exists a $n\times n$ Hermitian positive semidefinite matrix $Q$ such that
\begin{equation}
 \label{eq:bsos}
 f(x) = [b(x)]^* Q [b(x)] = \sum_{1\leq i,j \leq n} Q_{ij} \overline{b_i(x)} b_j(x) \quad \forall x \in G
\end{equation}
where $[b(x)] := [b_i(x)]_{i=1,\dots,n} \in \CC^n$. If \eqref{eq:bsos} holds where $Q$ is Hermitian positive semidefinite, we say that $Q$ is a \emph{Gram matrix} for $f$ in the basis $b_1,\dots,b_n$.
\end{prop}
\begin{proof}
Assume first that $f$ is a sum of squares, i.e., $f(x) = \sum_{k=1}^K |f_k(x)|^2$. Since $(b_1,\dots,b_n)$ forms a basis of $\cF(G,\CC)$ we can write $f_k(x) = \sum_{i=1}^n \bar{a_{ki}} b_i(x)$ for some coefficients $a_{ki} \in \CC$. Note that
$|f_k(x)|^2 = \sum_{1\leq i, j \leq n} a_{ki} \overline{a_{kj}} \overline{b_i(x)} b_j(x)$
and thus
$f(x) = \sum_{k} |f_k(x)|^2 = \sum_{1\leq i, j \leq n} Q_{i,j} \overline{b_i(x)} b_j(x)$
where $Q$ is the Hermitian matrix defined by:
$Q_{i,j} = \sum_{k} a_{ki} \overline{a_{kj}}$.
Note that $Q$ is positive semidefinite since it has the form $Q = \sum_{k} a_k a_k^*$ where $a_k$ is the vector $(a_k)_i = a_{ki}$.

We now show the converse. Assume $f$ can be written as \eqref{eq:bsos}. Since $Q$ is positive semidefinite, we can find vectors $a_k$ such that $Q = \sum_{k=1}^K a_k a_k^*$. If we define $f_k$ to be the function $f_k(x) = \sum_{i=1}^n \bar{a_{ki}} b_i(x)$ then we can verify that $f = \sum_{k=1}^K |f_k|^2$.
\end{proof}
Given $y \in G$ define the Dirac function $\delta_y$ at $y$ by:
\[ \delta_y(x) = \begin{cases} 1 & \text{ if } x = y\\ 0 & \text{ else.} \end{cases} \]
Then it is easy to see that we have:
\begin{prop}
\label{prop:nonnegative=sos}
Any nonnegative function $f$ on $G$ has a sum-of-squares certificate as:
\begin{equation}
\label{eq:diracsos}
 f(x) = \sum_{y \in G} |\sqrt{f(y)} \delta_y(x)|^2 \quad \forall x \in G.
\end{equation}
\end{prop}
Said differently, a nonnegative function $f$ is a sum-of-squares because if we pick $b_1,\dots,b_n$ to be the basis of Dirac functions, then $f$ satisfies Equation \eqref{eq:bsos} where $Q$ is the diagonal matrix consisting of the values taken by $f$ on $G$.

Since we are working with functions on a finite abelian group $G$, it is more natural (and more beneficial, as we see later) to look at sum-of-squares representation in the basis of \emph{characters}. One reason for this is that typically the functions $f$ we are interested in have a small support in the basis of characters and in this case one can hope to find a sum-of-squares decomposition which also only involves a small number of characters.
%, and so we are interested in the question of whether we can find a sum-of-squares decomposition that is also sparse.
The next theorem is simply a change-of-basis in the formula \eqref{eq:diracsos}.
\begin{prop}
\label{prop:soschar}
Let $f:G\rightarrow \RR$ and assume that $f$ is nonnegative, i.e., $f(x) \geq 0$ for all $x \in G$. Define the Hermitian matrix $Q \in \RR^{|G|\times |G|}$ indexed by characters $\chi\in \hat{G}$ by:
\begin{equation}
\label{eq:Q}
 Q_{\chi,\chi'} = \hat{f}(\overline{\chi} \chi').
\end{equation}
Then $Q$ is positive semidefinite and we have for any $x \in G$:
\begin{equation}
\label{eq:id}
 f(x) = \frac{1}{|G|} [\chi(x)]^* Q [\chi(x)] = \frac{1}{|G|} \sum_{\chi,\chi' \in \hat{G}} Q_{\chi,\chi'} \overline{\chi(x)} \chi'(x).
\end{equation}
\end{prop}
\begin{proof}
%\red{Write proof so that it has nice dual version}
Consider the matrix $X = [\chi(x)]_{x \in G, \chi \in \hat{G}}$ where rows are indexed by elements $x \in G$ and columns are indexed by characters $\chi \in \hat{G}$. Since the characters form an orthonormal basis of $\cF(G,\CC)$ for the inner product \eqref{eq:ip}, this means that the matrix $\frac{1}{\sqrt{|G|}} X$ is a unitary matrix. Note that we can rewrite the definition \eqref{eq:Q} of $Q$ in matrix terms as follows:
\[ Q = \frac{1}{|G|} X^* \diag( [f(x)]_{x \in G} ) X, \]
where $\diag([f(x)]_{x \in G})$ is the diagonal matrix with the values $f(x)$ on the diagonal. This shows that the eigenvalues of $Q$ are the values $\{f(x), x \in G\}$, and thus $Q$ is positive semidefinite. Since $\frac{1}{\sqrt{|G|}} X$ is unitary we also get that:
\[ \diag( [f(x)]_{x \in G} ) = \frac{1}{|G|} X Q X^*. \]
which, when evaluated at the diagonal entries is exactly Equation \eqref{eq:id}.
\end{proof}

%\begin{rem}
%{\small 
%\red{Here is another proof without matrices, I am not sure which one we should include}.
%Consider the multiplication map $M_f : \cF(G,\CC) \rightarrow \cF(G,\CC)$ given by:
%\[ (M_f h)(x) = f(x) h(x) \quad \forall h \in \cF(G,\CC), \; \forall x \in G. \]
%Note that $M_f$ is diagonal in the basis of Dirac functions $(\delta_x)_{x \in G}$ and the eigenvalues are just the values $f(x)$. Thus $M_f$ is positive semidefinite. The main observation is that the matrix $Q$ given in \eqref{eq:Q} represents the linear map $M_f$ expressed in the basis of the characters. Indeed note that:
%\[
%\langle \chi, M_f \chi' \rangle = \frac{1}{N} \sum_{x \in G} \overline{\chi(x)} f(x) \chi'(x)  = \frac{1}{N} \sum_{x \in G} \overline{(\chi \overline{\chi'})(x)} f(x)  = \langle \chi \overline{\chi'}, f \rangle = \hat{f}(\chi \overline{\chi'}) = Q_{\chi,\chi'}. \]
%Thus this shows that the matrix $Q$ is positive semidefinite. To see that identity \eqref{eq:id} note that we have:
%\[
%\sum_{\chi,\chi' \in \hat{G}} \hat{f}(\bar{\chi} \chi') \overline{\chi(x)}
%\overset{(i)}{=} \sum_{\chi \in \hat{G}} \bar{\chi(x)} \sum_{\alpha \in \hat{G}} \hat{f}(\alpha) \alpha(x) \chi(x)
%\overset{(ii)}{=} \sum_{\chi,\alpha \in \hat{G}} \hat{f}(\alpha) \alpha(x) = |G| f(x)
%\]
%where in $(i)$ we used the change of variables $\alpha = \bar{\chi} \chi'$ and in $(ii)$ we used the fact that $\bar{\chi(x)}\chi(x) = 1$.
%Note that another way of deriving \eqref{eq:id} is to start from \eqref{eq:diracsos} and to express the functions $\delta_y$ in the Fourier basis.
%}
%\end{rem}

\begin{example}
\label{ex:Z6}
We now include a simple example to illustrate the previous theorem. Let $G = \ZZ_6$ and consider the function
\begin{equation}
 \label{eq:exampleZ6}
 f(x) = 1 - \frac{1}{2}(\chi_1(x) + \chi_{-1}(x)) = 1 - \cos(2 \pi x / 6) \quad \forall x \in \ZZ_6.
\end{equation}
Clearly $f(x) \geq 0$ for all $x \in \ZZ_6$. 
Also note that $\hat{f}(0) = 1$, $\hat{f}(1) = \hat{f}(-1) = -1/2$ and $\hat{f}(k) = 0$ for all $k \notin \{-1,0,1\}$. The matrix $Q$ defined in \eqref{eq:Q} associated to this function $f$ takes the form:
\begin{equation}
\label{eq:QZ6}
Q =
\begin{bmatrix}
1 & -1/2 & 0 & 0 & 0 & -1/2\\
-1/2 & 1 & -1/2 & 0 & 0 & 0\\
0 & -1/2 & 1 & -1/2 & 0 & 0\\
0 & 0 & -1/2 & 1 & -1/2 & 0\\
0 & 0 & 0 & -1/2 & 1 & -1/2\\
-1/2 & 0 & 0 & 0 & -1/2 & 1
\end{bmatrix}
\end{equation}
\end{example}

%\subsubsection{Functions with a sparse support and chordal covers}

%The functions $f$ of interest are often supported on a small set of Fourier coefficients. For instance, in Example \ref{ex:Z6}, the function $f$ was supported only on the characters $\{1,\chi_{1},\chi_{-1}\}$. In this section we see how this information can be used to obtain \emph{sparse} sum-of-squares decompositions.

We are now interested in nonnegative functions $f:G\rightarrow \RR$ that are supported on a subset $\cS \subseteq \hat{G}$, i.e., $\hat{f}(\chi) = 0$ for all $\chi \notin \cS$. For such functions we are interested in finding \emph{sparse} sum-of-squares certificates for $f$, i.e., we are interested in finding a set $\cT \subseteq \hat{G}$ such that any nonnegative function $f$ supported on $\cS$ has a sum-of-squares certificate of the form:
%We know from the previous discussion that if $f$ is nonnegative, then $f$ can be written as a sum-of-squares of functions on $G$. In the rest of this section we are interested in finding 
%In this section we are interested in nonnegative functions $f:G\rightarrow \RR$ that are supported on a subset $\cS \subseteq \hat{G}$, i.e., $\hat{f}(\chi) = 0$ for all $\chi \notin \cS$. We know from the previous section that if $f$ is nonnegative, then $f$ can be written as a sum-of-squares of functions on $G$. In this section we show that when $f$ is supported on $\cS$, then it can be written as a sum-of-squares of functions supported on some particular set $\cT \subseteq \hat{G}$ constructed from $\cS$, i.e.,
\begin{equation}
 \label{eq:sparsesosT}
 f = \sum_{k=1}^K |f_k|^2 \quad \text{ where } \quad \supp f_k \subseteq \cT \;\; \; \forall k=1,\dots,K.
\end{equation}

The main idea to obtain such a ``sparse'' sum-of-squares certificate of $f$ is to exploit the sparsity of the Gram matrix $Q$ from Proposition \ref{prop:soschar}. Indeed, note that if $\supp f = \cS$, then the Gram matrix $Q$ of Proposition \ref{prop:soschar} has a specific sparsity structure:
%we know from Proposition \ref{prop:soschar} that if $f:G\rightarrow \RR$ is a nonnegative function, then we can write $f$ as a sum-of-squares in the basis of characters where the Gram matrix $Q$ is given by:
%\[ Q_{\chi,\chi'} = \hat{f}(\bar{\chi} \chi') \quad \forall \chi, \chi'\in \hat{G}. \]
%When $f$ is supported on $\cS$ the matrix $Q$ has a specific sparsity structure. Indeed, note that
\[ Q_{\chi,\chi'} \neq 0 \Leftrightarrow \bar{\chi} \chi' \in \cS. \]
In other words, the sparsity structure of $Q$ is given by the \emph{Cayley graph} $\Cay(\hat{G},\cS)$. Recall the definition of a Cayley graph:
\begin{defn}
Let $H$ be a group and let $\cS \subset H$ be a subset of $H$ that is symmetric, i.e., $x\in \cS \Rightarrow x^{-1} \in \cS$. The Cayley graph $\Cay(H,\cS)$ is the graph where vertices are the elements of the group $H$, and where two distinct vertices $x,y \in H$ are connected by an edge if $x^{-1} y \in \cS$ (or $y^{-1} x \in \cS$, which is the same since $\cS$ is symmetric).
\end{defn}
\begin{rem}
We do not require the set $\cS$ to be a generator for the group $H$ and hence the graph $\Cay(H,\cS)$ may be disconnected. Also observe that the set $\cS = \supp f$ in our case is symmetric since $f$ is real-valued; indeed when $f$ is real-valued we have $\hat{f}(\bar{\chi}) = \bar{\hat{f}(\chi)}$ for all $\chi \in \hat{G}$ and thus $\chi \in \supp f \Rightarrow \bar{\chi} \in \supp f$.
\end{rem}

To obtain a set $\cT \subseteq \hat{G}$ such that \eqref{eq:sparsesosT} holds for all functions $f$ supported on $\cS$ we will study \emph{chordal covers} of the graph $\Cay(\hat{G},\cS)$. We now introduce the key definition of \emph{Fourier support} for a graph with vertices $\hat{G}$.
\begin{defn}
\label{def:fouriersupportgraph}
Let $\Gamma$ be a graph with vertices $\hat{G}$. We say that $\Gamma$ has \emph{Fourier support} (or \emph{frequencies}) $\cT \subseteq \hat{G}$ if for any maximal clique $\cC$ of $\Gamma$ there exists $\chi_{\cC} \in \hat{G}$ such that $\chi_{\cC} \cC \subseteq \cT$ (where $\chi_{\cC} \cC:=\{\chi_{\cC} \chi : \chi \in \cC\}$ is the translation of $\cC$ by $\chi_{\cC}$).
\end{defn}
%\begin{rem}
%Observe that if $\Gamma$ has Fourier support $\cT$, then $\chi \cT := \{\chi t : t \in \cT\}$ is also a valid Fourier support for $\Gamma$. In general, there is no unique Fourier support associated to a graph $\Gamma$.
%\end{rem}
%\cbstart
Note that one can also state the definition of Fourier support of $\Gamma$ in terms of \emph{equivalence classes} of cliques:  Given a subset $\cC \subseteq \hat{G}$ define the \emph{equivalence class} of $\cC$ to be all the subsets of $\hat{G}$ that can be obtained from $\cC$ by translation, i.e., it is the set $[\cC] := \{ \chi \cC : \chi \in \hat{G}\}$. Using this terminology, the graph $\Gamma$ has Fourier support $\cT$ if for any maximal clique $\cC$ of $\Gamma$ there is at least one representative from $[\cC]$ that is contained in $\cT$.
%\cbend
% chordal cover}:
%\begin{defn}
%\label{def:triangulationfrequencies}
%Let $\cS$ be a symmetric subset of $\hat{G}$ (i.e., $\chi \in \cS \Rightarrow \chi^{-1} \in \cS$) and consider the Cayley graph $\Cay(\hat{G},\cS)$. We say that $\Cay(\hat{G},\cS)$ has a \emph{chordal cover (or triangulation) with frequencies $\cT \subseteq \hat{G}$} if there exists a chordal cover $\Gamma$ of $\Cay(\hat{G},\cS)$ such that the following is true: For any maximal clique $\cC$ of $\Gamma$ there exists $\chi_{\cC} \in \hat{G}$ such that $\chi_{\cC} \cC \subseteq \cT$, where $\chi_{\cC} \cC:=\{\chi_{\cC} \chi : \chi \in \cC\}$ is the translation of $\cC$ by $\chi_{\cC}$.
%\end{defn}

%We now state and prove our main theorem, Theorem \ref{thm:intro-main-sos} from the introduction (we use the same numbering since it is simply a restatement).
%\cbstart
We are now ready to state and prove our main theorem (the theorem below was stated as Theorem \ref{thm:intro-main-sos} in the introduction and we reuse the same numbering here since it is just a restatement).
%\cbend
%we reuse the same numbering since it is a simply a restatement of Theorem \ref{thm:intro-main-sos} from the introduction).
\begin{repthm}{thm:intro-main-sos}
\label{thm:mainsos}
Let $\cS$ be a symmetric subset of $\hat{G}$ and assume that $\Cay(\hat{G},\cS)$ has a chordal cover with Fourier support $\cT \subseteq \hat{G}$. Then any nonnegative function supported on $\cS$ admits a sum-of-squares certificate supported on $\cT$.
\end{repthm}
\begin{proof}
Let $f:G\rightarrow \RR$ be a nonnegative function supported on $\cS$. Let $Q$ be the Gram matrix \eqref{eq:Q} associated to the sum-of-squares representation of $f$ in the basis of characters. We saw that $Q$ is sparse according to the Cayley graph $\Cay(\hat{G},\cS)$. Since $\Gamma$ is a cover of $\Cay(\hat{G},\cS)$, $Q$ is also sparse according to $\Gamma$. Thus, since $\Gamma$ is chordal, using Theorem \ref{thm:sparsematrixdecomposition} we can find a decomposition of $Q$ as follows:
\begin{equation}
 \label{eq:pfQdecomp}
 Q = \sum_{\cC} Q_{\cC}
\end{equation}
where the sum is over the maximal cliques $\cC$ of $\Gamma$ and where each $Q_{\cC}$ is a positive semidefinite matrix supported on $\cC$. Note that Equation \eqref{eq:pfQdecomp} implies that for all $x \in G$:
\[
[\chi(x)]^* Q [\chi(x)] = \sum_{\cC} [\chi(x)]^* Q_{\cC} [\chi(x)],
\]
where $[\chi(x)] := [\chi(x)]_{\chi \in \hat{G}}$. Since $f(x) = [\chi(x)]^* Q [\chi(x)] / |G|$ the above equation says that:
\[ f(x) = \sum_{\cC} f_{\cC}(x) \]
where we let $f_{\cC}(x) := [\chi(x)]^* Q_{\cC} [\chi(x)] / |G|$.
Since $Q_{\cC}$ is positive semidefinite and supported on $\cC$, this means that each $f_{\cC}(x)$ is a sum-of-squares of functions supported on $\cC \subseteq \hat{G}$, i.e.,
\[ f_{\cC} = \sum_{k} |f_{\cC,k}|^2 \]
where $\supp f_{\cC,k} \subseteq \cC$.

According to Definition \ref{def:fouriersupportgraph}, we know that there exist $\chi_{\cC} \in \hat{G}$ for each maximal clique $\cC$ of $\Gamma$ such that $\chi_{\cC} \cC \subseteq \cT$. Now, observe that:
\[
f = \sum_{\cC} f_{\cC} = \sum_{\cC} \sum_{k} |f_{\cC,k}|^2 \overset{(i)}{=} \sum_{\cC} \sum_{k} |\chi_{\cC} f_{\cC,k}|^2
\overset{(ii)}{=} \sum_{\cC} \sum_{k} |\tilde{f}_{\cC,k}|^2
\]
where in $(i)$ we used the fact that $|\chi_{\cC}|^2 = 1$ and in $(ii)$ we let $\tilde{f}_{\cC,k} = \chi_{\cC} f_{\cC,k}$ which is supported on $\chi_{\cC} \cC \subseteq \cT$.
Thus we have shown that $f$ is a sum-of-squares of functions supported on $\cT$.
\end{proof}

\begin{example}
    \label{eg:hexagonsos}
Let $G = \ZZ_6$ and let $\cS = \{-1,0,1\} \subset \hat{\ZZ_6}$. We will use the previous theorem to show that any nonnegative function on $\ZZ_6$ supported on $\cS=\{-1,0,1\}$ is a sum-of-squares of functions supported on $\cT = \{-1,0,1,3\} \subseteq \hat{\ZZ_6}$. 
%\marginparsmall{Shall we index the vertices of the graph with $0,\dots,5$ or $-2,\dots,3$?}
The Cayley graph $\Cay(\hat{\ZZ_6},\{-1,0,1\})$ is the cycle graph on 6 nodes shown in Figure \ref{fig:example_hexagon}(left). Clearly the graph is not chordal since the cycle $0,1,\dots,5$ has no chord. Figure \ref{fig:example_hexagon}(right) shows a chordal cover $\Gamma$ of $\Cay(\hat{\ZZ_6},\{-1,0,1\})$ where the maximal cliques are: 
\[ \cC_1 = \{0,1,3\}, \;\; \cC_2 = \{1,2,3\}, \;\; \cC_3 = \{3,4,5\}, \;\; \cC_4 = \{0,3,5\}. \]
\begin{figure}[ht]
  \centering
  \includegraphics[scale=1]{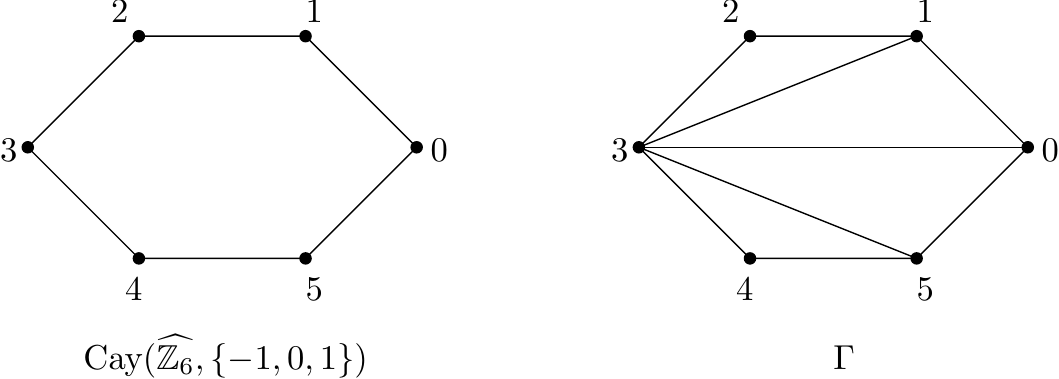}
  \caption{Left: The Cayley graph $\Cay(\hat{\ZZ_6},\{-1,0,1\})$ is the cycle graph on 6 nodes. Right: A chordal cover of the cycle graph, $\Gamma$.}
  \label{fig:example_hexagon}
\end{figure}

Observe that if we translate the clique $\cC_2 = \{1,2,3\}$ by $-2$ we get $\{-1,0,1\}$ and similarly if we translate the clique $\{3,4,5\}$ by $-4$ we also get $\{-1,0,1\}$. Thus by choosing 
\[ \chi_{\cC_1} = 0, \;\; \chi_{\cC_2} = -2, \;\; \chi_{\cC_3} = -4, \;\; \chi_{\cC_4} = 0 \] we get that $\chi_{\cC} + \cC \subseteq \{-1,0,1,3\}$ for all maximal cliques $\cC$ of $\Gamma$ (we used the fact that $5 = -1$ in $\ZZ_6$). In other words we have shown that $\Gamma$ is a chordal cover of $\Cay(\hat{\ZZ_6},\{-1,0,1\})$ with frequencies $\{-1,0,1,3\}$. Thus by Theorem \ref{thm:mainsos}, this means that any nonnegative function on $\ZZ_6$ supported on $\{-1,0,1\}$ can be written as a sum-of-squares of functions supported on $\{-1,0,1,3\}$.
%If $f$ is a nonnegative function on $\ZZ_6$ supported on $\cS = \{-1,0,1\}$, then the Gram matrix $Q$ of Theorem \ref{thm:soschar} for $f$ has the form:
%\[
%Q =
%\begin{bmatrix}
%a_0 & a_1 & 0 & 0 & 0 & \bar{a_1}\\
%\bar{a_1} & a_0 & a_1 & 0 & 0 & 0\\
%0 & \bar{a_1} & a_0 & a_1 & 0 & 0\\
%0 & 0 & \bar{a_1} & a_0 & a_1 & 0\\
%0 & 0 & 0 & \bar{a_1} & a_0 & a_1\\
%a_1 & 0 & 0 & 0 & \bar{a_1} & a_0
%\end{bmatrix}
%\]
%Here $a_0 = \hat{f}(0)$ and $a_1 = \hat{f}(1) = \bar{\hat{f}(-1)}$. Note that the sparsity structure of $Q$ is given by the cycle graph $\Cay(\hat{\ZZ_6},\{-1,0,1\})$. Since $\Gamma$ is a chordal cover of $\Cay(\hat{\ZZ_6},\{-1,0,1\})$, Theorem \ref{thm:sparsematrixdecomposition} says that we can decompose the matrix $Q$ as:
%\[
%Q = 
%{\small
%\begin{bmatrix}
%\times & \times & \times &  &  & \\
%\times & \times & \times &  &  & \\
%\times & \times & \times &  &  & \\
% &  &  &  &  & \\
% &  &  &  &  & \\
% &  &  &  &  & 
%\end{bmatrix}
%+
%\begin{bmatrix}
% &  &  &  &  & \\
% &  &  &  &  & \\
% &  & \times & \times & \times & \\
% &  & \times & \times & \times & \\
% &  & \times & \times & \times & \\
% &  &  &  &  & 
%\end{bmatrix}
%+
%\begin{bmatrix}
%\times &  &  &  & \times & \times\\
% &  &  &  &  & \\
% &  &  &  &  & \\
% &  &  &  &  & \\
%\times &  &  &  & \times & \times \\
%\times &  &  &  & \times & \times
%\end{bmatrix}
%+
%\begin{bmatrix}
%\times &  & \times &  & \times & \\
% &  &  &  &  & \\
%\times &  & \times &  & \times & \\
% &  &  &  &  & \\
%\times &  & \times &  & \times & \\
% &  &  &  &  & 
%\end{bmatrix}
%}
%\]
\end{example}

\subsection{Dual point of view: moment matrices and matrix completion}
\label{sec:moment}
% !TEX root = chordal_completion_paper.tex

Section~\ref{sec:sos} shows that nonnegative functions $f:G\rightarrow \RR$ that 
have Fourier support $\cS$ can be written as sums of Hermitian squares of 
functions $h:G\rightarrow \CC$ with Fourier support $\cT$, where $\cT$ satisfies 
a combinatorial property related to the Cayley graph $\Cay(\hat{G},\cS)$. 
In this section we describe the same results from the dual point of view,
arriving at a dual statement (Theorem~\ref{thm:moment-main})
of the main result of Section~\ref{sec:sos} (Theorem~\ref{thm:mainsos}).
The dual result describes a certain moment polytope $\cM(G,\cS)$ (see Definition~\ref{def:moment-polytope} to follow)
as the projection of a section of the cone of positive semidefinite matrices indexed by $\cT$. 
We could obtain this dual result by applying a conic duality argument directly to the statement 
of Theorem~\ref{thm:mainsos}. The purpose of this section, however, 
is to re-explain the results of Section~\ref{sec:sos} from an alternative viewpoint.

\subsubsection{Moment polytopes}
We begin by describing the moment polytope $\cM(G,\cS)$ where $\cS\subseteq \hat{G}$ is a collection of characters.
Concretely $\cM(G,\cS)$ is the convex hull of a collection of complex vectors indexed by $\cS$:
\[ \cM(G,\cS) = \textup{conv}\{(\chi(x))_{\chi\in \cS}\in \CC^{\cS}: x\in G\}.\]
The following definition is equivalent, somewhat easier to work with, and more readily generalizable.
\begin{defn}
    \label{def:moment-polytope}
    The \emph{moment polytope} $\cM(G,\cS)$ with respect to the characters $\cS\subseteq \hat{G}$, is the convex polytope
\[ \cM(G,\cS) = \{ (\EE_{\mu}[\chi])_{\chi\in \cS}\in \CC^{\cS}: \textup{$\mu$ a probability measure supported on $G$}\}.\]
\end{defn}
% not quite the polar...
% This can be thought of as the polar of nonnegative functions with Fourier support $\cS$, that is 
% the set of linear functionals on $\cF(G,\RR)$ taking values at most one on 
% nonnegative functions with Fourier support $\cS$.
For example, in the case where $G = \ZZ_6$ and $\cS = \{-1,1\}$, the moment polytope is 
\[ \cM(\ZZ_6,\{-1,1\}) = \textup{conv}\{(e^{\frac{-2\pi ki}{6}},e^{\frac{2\pi ki}{6}}): k\in \ZZ_6\}.\]
% Under the invertible $\RR$-linear map $(z,\bar{z})\mapsto (\Re(z),\Im(z))$ this is just the regular hexagon in $\RR^2$
% \[ \textup{conv}\{\left(\cos\left(\textstyle{\frac{2\pi k}{6}}\right),\sin\left(\textstyle{\frac{2\pi k}{6}}\right)\right): k\in \ZZ_6\}.\]

Nonnegative functions on finite abelian groups are sums of squares. Furthermore, we can express sums of squares 
concretely in terms of a Gram matrix. There is a similarly concrete way to 
describe the constraints that must be satisfied by a collection of complex numbers $\ell\in \CC^{\hat{G}}$
if they are a valid collection of moments of a probability measure supported on $G$. This description is given 
naturally in terms of a matrix constructed from $(\ell_\chi)_{\chi\in \hat{G}}$.
\begin{defn}
    \label{def:moment-matrix}
    If $\ell \in \CC^{\hat{G}}$, the associated \emph{moment matrix} is the 
    square matrix with rows and columns indexed by $\hat{G}$ of the form
    \[ [M(\ell)]_{\chi,\chi'} = \ell_{\bar{\chi}\chi'}\quad\text{for all $\chi,\chi'\in \hat{G}$}.\]
    If $\cT\subseteq \hat{G}$ and $\ell\in \CC^{\cT^{-1}\cT}$, the associated
    \emph{truncated moment matrix} is the square matrix with rows and columns indexed by $\cT$ of the form
    \[ [M_{\cT}(\ell)]_{\chi,\chi'} = \ell_{\bar{\chi}\chi'}\quad\text{for all $\chi,\chi'\in \cT$}.\]
\end{defn}
We now describe the dual version of the fact that any nonnegative function on $G$ is a 
sum of squares. Writing this in coordinates gives a concrete description in terms of moment matrices.
That probability measures are real-valued and nonnegative, and have total mass one corresponds to 
the conditions that for any valid moment vector $\ell$, the moment matrix $M(\ell)$ is Hermitian and positive semidefinite, and has
unit diagonal.
% changed the numbering to be a usual proposition
\begin{prop}
%\begin{dualprop}{prop:nonnegative=sos}
    \label{lem:moment-obv}
    The moment polytope can be expressed as 
    \[\cM(G,\hat{G}) = \{(\ell(\chi))_{\chi\in \hat{G}}: \ell(1_{\hat{G}})=1,\;\;\ell(|f|^2) \geq 0\quad\text{for all $f \in \cF(G,\CC)$}\}.\]
    Equivalently, defining coordinates $\ell_\chi := \ell(\chi)$ with respect to the character basis for $\cF(G,\CC)$ we have
\[ \cM(G,\hat{G}) = \{\ell\in \CC^{\hat{G}}: \ell_{1_{\hat{G}}} = 1,\;\; M(\ell)\psd 0\}.\]
%[\ell_{\bar{\chi}\chi'}]_{\chi,\chi'\in \hat{G}} \psd 0\}.\]
%\end{dualprop}
\end{prop}
\begin{proof}
    See Appendix~\ref{app:moment}
\end{proof}
 Some readers may recognize this as essentially a statement of 
Bochner's theorem for finite abelian groups~\cite{rudin1990fourier}. If we regard $\ell$ as a function $\ell:\hat{G}\rightarrow \CC$, 
the condition that $M(\ell) \psd 0$ is exactly saying that $\ell$ is a \emph{positive definite function} on $\hat{G}$. 
While we could use the language of positive definite functions throughout this section, we instead use the more concrete
language of moment matrices. We do this both so that our descriptions are compatible with the literature on polynomial optimization,
and so that they are easy to implement in code. 

The polytope $\cM(G,\cS)$ is just the projection of $\cM(G,\hat{G})$ onto the coordinates cooresponding to $\cS$.  
Alternatively we can think of $\cM(G,\cS)$ as those points in $\CC^{\cS}$ that can be completed to valid moment sequences. 
This suggests describing $\cM(G,\cS)$ in terms of a structured positive semidefinite matrix completion problem.
In this problem the diagonal is given, the entries corresponding to the edges of $\Cay(\hat{G},\cS)$ are given, 
and the goal is to complete the matrix to a positive semidefinite moment matrix.
\begin{cor}
    \label{cor:moment-structured}
    The moment polytope with respect to $\cS$ can be expressed as
    \begin{equation}
    \label{eq:moment-structured}
    \cM(G,\cS) = \{\ell\in \CC^{\cS}: \exists y\in \CC^{\hat{G}} \;\;\text{s.t.}\;\;
\ell_\chi = y_\chi\;\;\text{for all $\chi\in \cS$},\;\;y_{1_{\hat{G}}} =1,\;\; M(y)\psd 0\}.
%[y_{\bar{\chi}\chi'}]_{\chi,\chi'\in \hat{G}} \psd 0\}.
\end{equation}
\end{cor}
In the example of $\cM(\ZZ_6,\{-1,1\})$ we index characters by $\{0,1,2,3,4,5\}$ (so that $-1=5$). Then 
Corollary~\ref{cor:moment-structured} tells us that $\cM(\ZZ_6,\{-1,1\})$ is the set of $(\ell_5,\ell_{1})$ such that 
\[ \exists y_2,y_{3},y_4\in \CC\;\;\text{s.t.}\;\;  
    \begin{bmatrix} 1 & \ell_1 & y_2 & y_3 & y_{4} & \ell_{5}\\
    \ell_{5} & 1 & \ell_1 & y_2 & y_3 & y_{4}\\
    y_{4} & \ell_{5} & 1 & \ell_1 & y_2 & y_3\\
    y_3 & y_{4} & \ell_{5} & 1 & \ell_1 & y_2\\
    y_2 & y_3 & y_{4} & \ell_{5} & 1 & \ell_1\\
\ell_1 & y_2 & y_3 & y_{4} & \ell_{5} & 1\end{bmatrix} \psd 0.\]
%\cbstart
Note that we adopt the convention, throughout, that writing $M(y) \psd 0$ implies that $M(y)$ is Hermitian. This may, in turn, 
imply certain linear equalities on $y$. For instance, in the above example the notation implies 
that $\ell_1 = \bar{\ell_5} = \bar{\ell_{-1}}$ and $y_2 = \bar{y_4} = \bar{y_{-2}}$ and $y_3 = \bar{y_3}$.
%\cbend 

\subsubsection{Moment polytopes and matrix completion}

In Proposition~\ref{prop:soschar} we saw that any nonnegative (and hence sum of squares) function with Fourier support $\cS$ 
has a Gram matrix that is sparse with respect to the graph $\Cay(\hat{G},\cS)$. The dual statement is that the moment 
polytope $\cM(G,\cS)$ can be described in terms of an \emph{un}structured positive semidefinite matrix completion problem. In this 
case we are given the diagonal and the entries correponding to the edges of $\Cay(\hat{G},\cS)$ and just need to complete the 
matrix to be positive semidefinite, without requiring it to be a moment matrix.
% changed the numbering to be a usual proposition
%\begin{dualprop}{prop:soschar}
\begin{prop}
    \label{prop:moment-unstructured}
    The moment polytope with respect to $\cS$ can be expressed as
    \[ 
        \cM(G,\cS) = \{\ell\in \CC^{\cS}: \exists Y\in \HH_+^{\hat{G}}\;\;\text{s.t.}\;\;
Y_{\chi,\chi'} = \ell_{\bar{\chi}\chi'}\;\;\text{whenever $\bar{\chi}\chi'\in \cS$ and $Y_{\chi,\chi}=1$ for all $\chi\in \hat{G}$}\}.\]
%\end{dualprop}
\end{prop}
\begin{proof}[Idea of proof]
    We describe the main idea of the proof, giving the details in Appendix~\ref{app:moment}.
    The key issue is to show that if $\ell\in \CC^{\cS}$ has a completion to a positive semidefinite matrix $Y\in \HH_+^{\hat{G}}$ then $\ell$ 
    also has a completion to a positive semidefinite \emph{moment} matrix $M(y)$ (for some $y\in \CC^{\hat{G}}$). 
    This can be established by observing that 
    the group $\hat{G}$ acts on the rows and columns of Hermitian matrices $\HH^{\hat{G}}$ by permutations. This action fixes, 
    pointwise, positive semidefinite moment matrices.  Averaging the orbit of $Y$ under this group action gives a moment 
    matrix with the desired properties.
\end{proof}
Recall that Theorem~\ref{thm:mainsos} gave a combinatorial 
condition under which any nonnegative function with Fourier support $\cS$ is not just a sum of squares, but a sum of squares of 
functions with Fourier support $\cT$. The dual version says that under the same combinatorial condition, to check that $\ell\in \cM(G,\cS)$
we are not required to complete it to a full positive semidefinite moment matrix. Instead it is enough to complete it to a truncated moment matrix $M_{\cT}(y)$
for some $y\in \CC^{\cT^{-1}\cT}$. 
\begin{dualthm}{thm:mainsos}
    \label{thm:moment-main}
    Let $\cS$ be a symmetric subset of $\hat{G}$. If  $\Cay(\hat{G},\cS)$ has a chordal cover 
    with Fourier support $\cT\subseteq \hat{G}$ then 
    \begin{equation}
        \label{eq:moment-main}
        \cM(G,\cS) = \{\ell\in \CC^{\cS}: \exists y\in \CC^{\cT^{-1}\cT}\;\;\text{s.t.}\;\;
            y_\chi = \ell_\chi\;\;\text{for all $\chi\in \cS$},\;\;y_{1_{\hat{G}}}=1,\;\;\text{and}\;\;
        M_{\cT}(y) \psd 0\}.
\end{equation}
\end{dualthm}
% Before providing a proof, we point out the difference between the characterization of $\cM(G,\cS)$
% in Theorem~\ref{thm:moment-main} and Corollary~\ref{cor:moment-structured}. In Corollary~\ref{cor:moment-structured} to show that $(\ell_\chi)_{\chi\in \cS}\in \cM(G,\cS)$
% we had to complete it to a full, positive semidefinite, moment matrix. On the other hand, Theorem~\ref{thm:moment-main} tells 
% us that to show that $(\ell_\chi)_{\chi\in \cS}\in \cM(G,\cS)$ only need to be able to complete it to 
% the submatrix of the moment matrix indexed by $\cT$. 
\begin{proof}[Idea of proof]
    Again we summarize the key idea, deferring a detailed proof to Appendix~\ref{app:moment}. The main issue is to show that 
    being able to complete $\ell$ to a truncated moment matrix $M_{\cT}(y)$ for some $y\in \CC^{\cT^{-1}\cT}$
    implies we can complete it to a positive semidefinite matrix $Y\in \HH^{\hat{G}}_+$. Proposition~\ref{prop:moment-unstructured}
    then implies we can complete it to a positive semidefinite moment matrix. 

    The key observation, analogous to the translation of frequencies idea in the proof of Theorem~\ref{thm:mainsos}, 
    is that given $y\in \CC^{\cT^{-1}\cT}$, the partial moment matrices $M_{\chi \cT}(y)$ for any $\chi\in \hat{G}$ 
    are all the same. Hence imposing that $M_{\cT}(y) \psd 0$, implies that $M_{\chi\cT}(y) \psd 0$ for all $\chi\in \hat{G}$. 
    Since $\Cay(\hat{G},\cS)$ has a chordal cover $\Gamma$ with Fourier support $\cT$, we construct from $y$ a partial matrix supported
    on the maximal cliques of $\Gamma$. The conditions that $M_{\chi\cT}(y)\psd 0$ for all $\chi\in \hat{G}$ are enough to show
    that all the principal submatrices supported on maximal cliques of $\Gamma$ are positive semidefinite. The 
    chordal matrix completion result (Theorem~\ref{thm:completion}) completes the proof.
\end{proof}
\begin{example2}[Example~\ref{eg:hexagonsos} cont.]
Recall that the Cayley graph $\Cay(\hat{\ZZ_6},\{-1,1\})$ is the $6$-cycle. We label
the elements of $\hat{\ZZ_6}$ by $\{0,1,2,3,4,5\}$. The Cayley graph has a chordal cover (see Figure~\ref{fig:example_hexagon})
with Fourier support $\cT = \{-1,0,1,3\} = \{5,0,1,3\}$ (in $\ZZ_6$). Observe that $\cT^{-1}\cT = \hat{\ZZ_6}$.  
Applying Theorem~\ref{thm:moment-main} we see that 
\[ \cM(\ZZ_6,\{-1,1\}) = \left\{(\ell_5,\ell_{1}): \exists y_2,y_3,y_4\in \CC\;\;\text{s.t.}\;\;
        \begin{bmatrix} 1 & \ell_1 & y_3 & \ell_{5}\\
        \ell_{5} & 1 & y_2 & y_{4}\\
        y_3 & y_{4} & 1 & y_2\\
\ell_1 & y_2 & y_{4} & 1\end{bmatrix}\psd 0\right\}.\]
    \end{example2}
        % \todo{Move the end of this example to a 
        % section on real symmetric lifts}

        % Making the change of variables to $u_1 = (\ell_1+\ell_{-1})/2, v_1 = (\ell_1-\ell_{-1})/(2i), u_2 = -(y_2+y_{-2})/2, v_2 = (y_2-y_{-2})/(2i)$, and 
        % $v_3 = -y_3$ and noting that 
        % \[        
        %     \begin{bmatrix} 1 & 0 & 0 & 0\\0 & \frac{1}{2} & 0 & \frac{1}{2}\\0 & 0 & -1 & 0\\0 & \frac{i}{2} & 0 & -\frac{i}{2}\end{bmatrix}
        % \begin{bmatrix} 1 & \ell_1 & y_3 & \ell_{-1}\\
        % \ell_{-1} & 1 & y_2 & y_{-2}\\
        % y_3 & y_{-2} & 1 & y_2\\
% \ell_1 & y_2 & y_{-2} & 1\end{bmatrix}
        %     \begin{bmatrix} 1 & 0 & 0 & 0\\0 & \frac{1}{2} & 0 & \frac{1}{2}\\0 & 0 & -1 & 0\\0 & \frac{i}{2} & 0 & -\frac{i}{2}\end{bmatrix}^* 
        % = \begin{bmatrix} 1 & u_1 & v_3 & v_1\\u_1 & \frac{1-u_2}{2} & u_2 & \frac{v_2}{2}\\ v_3 & u_2 & 1 & v_2\\v_1 & \frac{v_2}{2} & v_2 & \frac{1+u_2}{2}\end{bmatrix}
        % \]
        % we recover the description of the regular hexagon from~\cite{fawzi2014equivariant} in terms of real symmetric matrices as
        % \[        \left\{(u_1,v_1): \exists u_2,v_2,v_3\in \RR\;\;\text{s.t.}\;\;
        % \begin{bmatrix} 1 & u_1 & v_3 & v_1\\u_1 & \frac{1-u_2}{2} & u_2 & \frac{v_2}{2}\\ v_3 & u_2 & 1 & v_2\\v_1 & \frac{v_2}{2} & v_2 & \frac{1+u_2}{2}\end{bmatrix}
% \psd 0\right\}.\]

\subsection{Real sums-of-squares and moment polytopes}
\label{sec:real}
% !TEX root = chordal_completion_paper.tex

The main results of Sections~\ref{sec:sos} and~\ref{sec:moment} work
with sums of Hermitian squares of complex-valued functions and complex Hermitian 
moment matrices respectively. While this is convenient mathematically, computationally
it is desirable to work with real-valued functions and real symmetric moment matrices.
In this section we give real versions of Theorem~\ref{thm:mainsos} and 
Theorem~\ref{thm:moment-main}.
These are the forms most suited to implementation and the forms
we use when discussing the examples in Sections~\ref{sec:maxcut} 
and~\ref{sec:cyclic} to follow.

\paragraph{Basic observations}
The main additional observation we make is that the dual group
$\hat{G}$ consists of two types of characters: those that are
real-valued, and those that are not. It is helpful to think
of this decomposition in terms of the involution $\chi \mapsto \chi^{-1}=\bar{\chi}$
on $\hat{G}$. Real-valued characters are those that are fixed by inversion (complex conjuation). 
The remaining characters come in inverse (complex conjugate) pairs. 

To fix notation let $\hat{G}_0$ denote 
the real-valued characters and $\hat{G}_{-1}\cup \hat{G}_{1}$ be a 
fixed partiton of the remaining characters into conjugate pairs. In particular $\hat{G}_{-1} = \hat{G}_1^{-1}$. 
If $\cS\subseteq\hat{G}$ is symmetric (i.e.\ $\cS^{-1} = \cS$) then 
inversion restricts to an involution on $\cS$ and so we have the decomposition
$\cS = \cS_0 \cup \cS_{-1} \cup \cS_{1}$ where $\cS_i = \hat{G}_i \cap \cS$ for $i=-1,0,1$. 

\paragraph{Sums-of-squares}
Suppose $f$ is a sum of Hermitian squares of functions $f_j:G\rightarrow \CC$, each 
supported on $\cT \subseteq \hat{G}$. Then $f$ can be expressed as a sum-of-squares
of real-valued functions $\Re[f_j]$ and $\Im[f_j]$ as
\[ f = \sum_j |f_j|^2 = \sum_j\left(\Re[f_j]^2 + \Im[f_j]^2\right).\]
Clearly $\Re[f_j]$ and $\Im[f_j]$ are supported on the symmetric subset 
$\cT \cup \cT^{-1}$ of characters. As such, the real analogue of Theorem~\ref{thm:mainsos}
is the following, the only modification being that we insist that $\cT$ is symmetric.
\begin{thm}
    Let $\cS\subseteq \hat{G}$ be symmetric and let $f:G\rightarrow \RR$ be a non-negative function supported on $\cS$. 
    If $\cT\subseteq \hat{G}$ is symmetric and $\Cay(\hat{G},\cS)$ has a chordal cover with Fourier support $\cT$ then 
    $f$ is a sum of squares of \emph{real-valued} functions supported on $\cT$. 
\end{thm}

\paragraph{Real moment matrices and moment polytopes}
We now develop the real analogue of moment polytopes and moment matrices. The
discussion is more explicit (and more involved) than was required for the sum-of-squares viewpoint. 
We begin by defining real moment polytopes. These are just linear transformations 
of the moment polytopes of Section~\ref{sec:moment}. Throughout this section for any symmetric $\cS\subseteq \hat{G}$
partitioned as $\cS_0 \cup \cS_{-1}\cup \cS_1$ fix a linear map $\cR_{\cS}: \CC^{\cS}\rightarrow \RR^{|\cS|}$
defined by
\[ \cR_{\cS}(\ell) = \left(\,(\ell_{\chi})_{\chi\in \cS_0}, (\Re[\ell_\chi],\Im[\ell_\chi])_{\chi\in \cS_1}\,\right).\]
Observe that $\cR_{\cS}$ depends on the partitioning of $\cS$.
\begin{defn}
    \label{def:real-moment}
    If $\cS\subseteq\hat{G}$ is symmetric and partitioned as $\cS_0\cup\cS_{-1}\cup\cS_{1}$, 
    the \emph{real moment polytope} with respect to $\cS$ is the image of $\cM(G,\cS)$ under $\cR_{\cS}$, i.e.\
    \[ \cM^{\RR}(G,\cS) = \cR_{\cS}(\cM(G,\cS)).\]
\end{defn} 
This is a polytope in $\RR^{|\cS|}$ that is affinely isomorphic to $\cM(G,\cS)$. 
    % \begin{align*}
    %     \cM^{\RR}(G,\cS) & = \conv\{((\chi(x))_{\chi\in\cS_0},(\Re[\chi](x),\Im[\chi](x))_{\chi\in \cS_1})\in \RR^{|\cS|}: x\in G\}\\
    %                      & = \{((\EE_\mu[\chi])_{\chi\in \cS_0},(\EE_\mu[\Re[\chi]],\EE_\mu[\Im[\chi]])_{\chi\in \cS_1})\in \RR^{|\cS|}: 
    % \textup{$\mu$ a probability measure supported on $G$}\}.
    % \end{align*}
% \end{def}
For example, in the case where $G = \ZZ_6$ and $\cS=\{-1,1\}$ is decomposed as $\cS_{-1} = \{-1\}$ and $\cS_1 = \{1\}$, 
the real moment polytope is 
\[ \cM^{\RR}(\ZZ_6,\{-1,1\}) = \conv\{(\cos(\textstyle{\frac{2\pi k}{6}}),\sin(\textstyle{\frac{2\pi k}{6}})): k\in \ZZ_6\},\]
the regular hexagon in the plane.

We have a description of $\cM(G,\cS)$, and hence of $\cM^{\RR}(G,\cS)$ in terms of Hermitian positive semidefinite matrices. Using this it is straightforward to 
give a description in terms of real symmetric positive semidefinite matrices
of twice the size (via~\eqref{eq:goemans} in Appendix~\ref{app:real}).
Our aim is to describe the real moment polytope $\cM^{\RR}(G,\cS)$ in terms of real symmetric 
positive semidefinite matrices \emph{without increasing the size of the description}. It turns out that we can 
do this whenever $\cT$ has a property that is related to, but less restrictive than, being symmetric.
\begin{defn}
    \label{defn:inv}
    A subset $\cT \subseteq \hat{G}$ has an \emph{equalizing involution} if there is an involution $\sigma: \cT \rightarrow \cT$ 
    such that $\sigma(\chi)\chi = \sigma(\chi')\chi'$ for all $\chi,\chi'\in \hat{G}$.
\end{defn}
Observe that if $\cT$ is symmetric, then the map that sends every element of $\cT$ to its inverse 
satisfies $\chi^{-1}\chi = (\chi')^{-1}\chi' = 1_{\hat{G}}$ for all $\chi,\chi'\in \cT$. Hence 
any symmetric subset of $\hat{G}$ has an equalizing involution. 
An example of a set that has an equalizing involution but is not symmetric is
$\cT = \{0,1,2,3\}\subseteq \hat{\ZZ_5}$. Then $\cT$ is not symmetric and, furthermore, 
there is no $k\in \hat{\ZZ_5}$ such that $k+\cT$ is symmetric. Nevertheless,
$\cT$ does have an equalizing involution given by $\sigma(0) = 3$, 
$\sigma(3)=0$, $\sigma(1)=2$ and $\sigma(2)=1$ since $\sigma(k)+k = 3$ for all $k\in \cT$. 

Our main result in this section is the following.
\begin{thm}
    \label{thm:moment-real}
    Let $\cS\subseteq\hat{G}$ be symmetric and let $\cT\subseteq\hat{G}$ have an equalizing involution $\sigma$. If $\Cay(\hat{G},\cS)$
    has a chordal cover with Fourier support $\cT$ then
    \begin{align*}
        \cM^{\RR}(G,\cS)  = \left\{\right.\cR_{\cS}(\ell)\in \RR^{|\cS|}: &\exists y\in \CC^{\cT^{-1}\cT}\;\text{s.t.}\;\;y_{\bar{\chi}} = \bar{y_\chi}\;\;
        \text{for all $\chi\in \cT^{-1}\cT$,}\\
        & \left.y_\chi = \ell_\chi\;\;\text{for all $\chi\in \cS$},\;\;y_{1_{\hat{G}}}=1,\;\;\text{and}\; 
    [\,\Re[y_{\bar{\chi}\chi'}] - \Im[y_{\sigma(\bar{\chi})\chi'}]\,]_{\chi,\chi'\in \cT} \psd 0\right\}.
    \end{align*}
\end{thm}
% \begin{thm}
%     \label{thm:moment-real}
%     Let $\cS\subseteq \hat{G}$ and $\cT\subseteq \hat{G}$ be symmetric. If $\Cay(\hat{G},\cS)$ has a chordal cover with frequencies
%     $\cT$ then
%     \begin{align*}
%         \cM^{\RR}(G,\cS) & = \{ (\;(u_\chi)_{\chi\in \cS_0},(u_\chi,v_\chi)_{\chi\in \cS_1}\;)\in \RR^{|\cS|}:
%         \exists (w_{\chi},x_\chi)_{\chi\in\bar{\cT}\cT},\;\;\text{s.t.}\;\;\\
%     & \qquad\qquad w_\chi = u_{\chi}\;\;\forall\chi\in (\bar{\cT}\cT)_0,\;\; w_\chi=u_\chi\;\;\forall\chi\in (\bar{\cT}\cT)_1, \;\;w_\chi = u_{\bar{\chi}}\;\;\forall\chi\in (\bar{\cT}\cT)_{-1},\\
%     & \qquad\qquad x_\chi = 0\;\quad\forall\chi\in(\bar{\cT}\cT)_0,\;\;x_\chi = v_\chi\;\;\forall\chi\in (\bar{\cT}\cT)_1,\;\; x_\chi = -v_{\bar{\chi}}\;\;
%         \forall\chi\in (\bar{\cT}\cT)_{-1},\\
%             & \qquad\qquad\qquad\qquad\qquad\qquad\qquad\qquad\qquad \qquad\qquad w_{1_{\hat{G}}}=1,\;\; [w_{\bar{\chi}\chi'} - x_{\chi\chi'}]_{\chi,\chi'\in\cT} \psd 0\}.
% \end{align*}
% \end{thm}
Note that the main change between Theorem~\ref{thm:moment-real} and Theorem~\ref{thm:moment-main}
is that we have replaced the condition that the (Hermitian) truncated moment matrix $M_\cT(y)$ be positive semidefinite
with the condition that a \emph{real symmetric} matrix indexed by $\cT$ be positive semidefinite. 
We also explicitly add the conjugate symmetry constraint that $y_{\bar{\chi}} = \bar{y_\chi}$, which was 
implied by $M_{\cT}(y)$ being Hermitian in Theorem~\ref{thm:moment-main}. 

\begin{example}
    Before giving a proof, we apply Theorem~\ref{thm:moment-real} to the case of
$\cM^{\RR}(\ZZ_6,\{-1,1\})$, the regular hexagon in the plane. Recall that in
this case we can take $\cT = \{0,1,3,5\}$ which is symmetric. Hence 
$\cT$ has $\sigma(\chi) = \chi^{-1}$ as an equalizing involution. Decomposing
$\hat{G}$ as $\hat{G}_0 = \{0,3\}$ and, for instance $\hat{G}_{1} = \{1,2\}$
and $\hat{G}_{-1} = \{4,5\}$ we have that $\cS_0 = \{0\}$ and $\cS_1 =
\{1\}$. Note also that $\cT^{-1}\cT = \hat{G}$ in this case, giving
$(\cT^{-1}\cT)_i = \hat{G}_i$ for $i=-1,0,1$.  Ordering the elements of $\cT$
as $(0,1,3,5)$ and writing $u_j = \Re[\ell_j]$, $v_j=\Im[\ell_j]$, $w_j=\Re[y_j]$ and $x_j=\Im[y_j]$.
we see that
\[ [\Re[y_{\bar{\chi}\chi'}] - \Im[y_{\sigma(\bar{\chi})\chi'}]]_{\chi,\chi'\in \cT} = [w_{\bar{\chi}\chi'} - x_{\chi\chi'}]_{\chi,\chi'\in \cT} = 
\begin{bmatrix} w_0-x_0 & w_1-x_1 & w_3-x_3 & w_{5}-x_{5}\\
w_{5}-x_1 & w_{0}-x_2 & w_2-x_{4} & w_{4} - x_0\\
w_3-x_3 & w_{4}-x_{4} & w_0-x_0 & w_2-x_2\\
w_1-x_{5}& w_2-x_0 & w_{4}-x_{2} & w_0 - x_{4}\end{bmatrix}.\]
    The conjugate symmetry constraint on $y$ implies that $x_0=x_3=0$, $w_1=w_5$, $x_1=-x_5$, $w_2=w_4$, and $x_2=-x_4$. Applying these
    and the constraints that $u_1=w_1$, $v_1=x_1$ and $w_0=1$, we
obtain a description of the regular hexagon in the plane as 
\[ \cM^{\RR}(\ZZ_6,\{-1,1\}) = \left\{(u_1,v_1): \exists w_2,w_3,x_2\in \RR\;\;\text{s.t.}\;\;
\begin{bmatrix} 1 & u_1-v_1 & w_3 & u_1+v_1\\
u_1-v_1 & 1-x_2 & w_2+x_2 & w_2\\
w_3 & w_2+x_2 & 1 & w_2-x_2\\
u_1+v_1 & w_2 & w_2-x_2 & 1+x_2\end{bmatrix} \psd 0\right\}.\]
    \end{example}
        \begin{proof}[Proof of Theorem~\ref{thm:moment-real}]
            % It is enough to show that $\cM(G,\cS)$ has a description in terms of a real symmetric linear matrix inequality as 
        % \[ \cM(G,\cS) = \{(\ell_\chi)_{\chi\in \cS}: \exists (y_\chi)_{\chi\in \bar{\cT}\cT}\;\;\text{s.t.}\;\;
        % y_{1_{\hat{G}}}=1,\;y_\chi = \ell_\chi\;\forall\chi\in \cS,\;\text{and}\;[\Re[y]_{\bar{\chi}\chi'} - \Im[y]_{\chi\chi'}]_{\chi,\chi'\in \cT} \psd 0\}.\]
        % The statement of Theorem~\ref{thm:moment-real} then follows by rewriting everything in terms of the variables $u_\chi = \Re[\ell_\chi]$, $v_\chi = \Im[\ell_\chi]$,
        % $w_\chi = \Re[y_\chi]$ and $x_\chi = \Im[y_\chi]$. 
        We use a novel result that allows us to write certain complex Hermitian 
    linear matrix inequalities as real symmetric linear 
    matrix inequalities of the same size. The key requirement is that 
    restricted to the subspace of Hermitian matrices of interest, 
    entry-wise complex conjugation can be expressed as congruence by an
    orthogonal symmetric matrix. 
\begin{lem}
    \label{lem:real}
    Let $\cL$ be a subspace (over the reals) of $\HH^d$. Suppose there is 
    some orthogonal $J\in O(d)$ such that $J^2=I$ and 
    \[ JLJ^T = \bar{L}\quad\text{for all $L\in \cL$,}\]
    i.e.\ congruence by $J$ restricted to $\cL$ is entry-wise 
    complex conjugation. Then 
    \[ \{ L\in \cL: L \in \HH^d_+\} = 
    \{ L\in \cL: \Re[L] - J\Im[L] \in \S_+^d\}.\]
\end{lem}
\begin{proof}
    See Appendix~\ref{app:real}.
\end{proof}
Let $\cL$ be the subspace (over the reals) of $\HH^{\cT}$ given by 
    \[ \cL = \{M(y): y\in \CC^{\cT^{-1}\cT},\;y_{\bar{\chi}} = \bar{y_\chi},\;\text{for all $\chi\in \cT^{-1}\cT$}\}.\]
    Let $J$ be the $|\cT|\times |\cT|$ permutation matrix representing the 
    equalizing involution $\sigma:\cT\rightarrow \cT$.  Since $\sigma$ is an involution, $J$ 
    satisfies $J^2=I$. The definition of equalizing involution comes from our desire that $\sigma$ satisfy the relation 
    $\sigma(\chi)\bar{\sigma(\chi')}= \bar{\chi}\chi'$, equivalent to the defining relation in 
    Definition~\ref{defn:inv}, for all $\chi,\chi'\in \cT$. Then for any conjugate symmetric $y$, 
    \begin{align*}
        JM(y)J  = [y_{\bar{\chi}\chi'}]_{\sigma(\chi),\sigma(\chi')\in \cT}& = [y_{\bar{\sigma(\chi)}\sigma(\chi')}]_{\chi,\chi'\in\cT}\\
              & = [\bar{y}_{\sigma(\chi)\bar{\sigma(\chi')}}]_{\chi,\chi'\in \cT}\quad\text{by conjugate symmetry of $y$}\\
              & = [\bar{y}_{\bar{\chi}\chi'}]_{\chi,\chi'\in \cT}\quad\text{since $\sigma$ is an equalizing involution.}
    \end{align*}
    Hence congruence by $J$ corresponds to entry-wise complex conjugation restricted to $\cL$.
    Applying Lemma~\ref{lem:real} we can conclude that for all $y\in \CC^{\cT^{-1}\cT}$,  
    \begin{align*}
        M(y)\psd 0 \quad& \iff\quad y_{\bar{\chi}} = \bar{y_\chi}\;\;\forall \chi\in \cT^{-1}\cT\;\;\text{and}\;\;
        % [\Re[y]_{\bar{\chi}\chi'} - J\Im[y]_{\bar{\chi}\chi'}]_{\chi,\chi'} \psd 0\\
        [\Re[y_{\bar{\chi}\chi'}] - \Im[y_{\sigma(\bar{\chi})\chi'}]]_{\chi,\chi'\in \cT} \psd 0.
    \end{align*}
    This completes the proof of Theorem~\ref{thm:moment-real}.
    \end{proof}

% ------
\section{Application 1: The case $G=\ZZ_2^n$ and binary quadratic optimization}
\label{sec:maxcut}
% !TEX root = chordal_completion_paper.tex

In this section we apply the results of Section~\ref{sec:sos} 
to the case
of nonnegative quadratic forms on the vertices of the hypercube
in $n$ dimensions. Dually, the moment polytope of interest in this 
section is the $n$th \emph{cut polytope} $\CUT_n$.
Our main aim is to establish Laurent's
conjecture~\cite[Conjecture 4]{laurent2003lower} that any nonnegative quadratic form 
on the vertices of the hypercube in $n$ dimension is a sum of squares 
of polynomials of degree at most $\lceil n/2\rceil$.

\subsection{Quadratic forms on $\{-1,1\}^n$ and the cut polytope}
Let $G = \{-1,1\}^n$
be the vertices of the hypercube in dimension $n$. View $G$ 
as a group (isomorphic to $\ZZ_2^n$) under componentwise multiplication.
Recall that the characters of $G$ are indexed 
by subsets $S\in 2^{[n]}$ and are the square-free monomials
\[ \chi_S(x) = \prod_{i\in S}x_i \quad\text{for all $x\in G$}.\]

We focus on characterizing nonnegative quadratic funcions on $G$. These
are of particular interest because the problem of maximizing a quadratic
form over $G$ i.e.\
\begin{equation}
    \label{eq:binary-quadratic}
    \max_{x\in G} \sum_{1\leq i<j\leq n}A_{ij} x_ix_j
\end{equation}
includes, for example, the \textsc{max-cut} problem, 
which arises when the symmetric matrix $A_{ij}$ is the Laplacian of 
a (weighted) graph on $n$ vertices. We can 
solve~\eqref{eq:binary-quadratic} by finding the smallest upper
bound on the objective:
\begin{equation}
    \label{eq:binary-quadratic-dual}
    \min_{\gamma} \gamma\quad\text{s.t.}\quad
    \gamma - \sum_{1\leq i<j\leq n}A_{ij}x_ix_j \geq 0\quad\text{for all
        $x\in G$.}
\end{equation}
If we have a characerization of nonnegative functions on $G$ with Fourier support 
$\cS = \{S\in 2^{[n]}: \text{$|S|=0$ of $|S|=2$}\}$ as sums of squares of 
functions supported on $\cT\subseteq \hat{G}$ then we can solve~\eqref{eq:binary-quadratic-dual}
by solving a semidefinite optimization problem of size $|\cT|$.

The dual picture to~\eqref{eq:binary-quadratic-dual} is to consider optimization over the 
moment polytope $\cM(\{-1,1\}^n,\cS\setminus\{\emptyset\})$, known as the \emph{cut polytope}
\[ \CUT_n:= \cM(\{-1,1\}^n,\cS\setminus\{\emptyset\}) = \conv\,\{(x_ix_j)_{1\leq i<j\leq n}: (x_1,x_2,\ldots,x_n)\in \{-1,1\}^n\}.\]
We can solve the binary quadratic optimization problem~\eqref{eq:binary-quadratic} by optimizing the 
linear function defined by $A$ over $\CUT_n$, i.e.\ by solving the linear program
\[ \max_{(\ell_{ij})_{1\leq i<j\leq n}} \sum_{i<j}\ell_{ij}A_{ij}\quad\text{s.t.}\quad (\ell_{ij})_{1\leq i<j\leq n} \in \CUT_n.\]
If we have a PSD lift of the cut polytope $\CUT_n$ of size $|\cT|$ then we can solve this optimization problem
by solving a semidefinite optimization problem of size $|\cT|$.

\subsection{The associated Cayley graph}
\label{sec:cayley-binary}
\begin{figure}
    \begin{center}
        \includegraphics[trim=3cm 19cm 2cm 2cm,clip, scale=0.8]{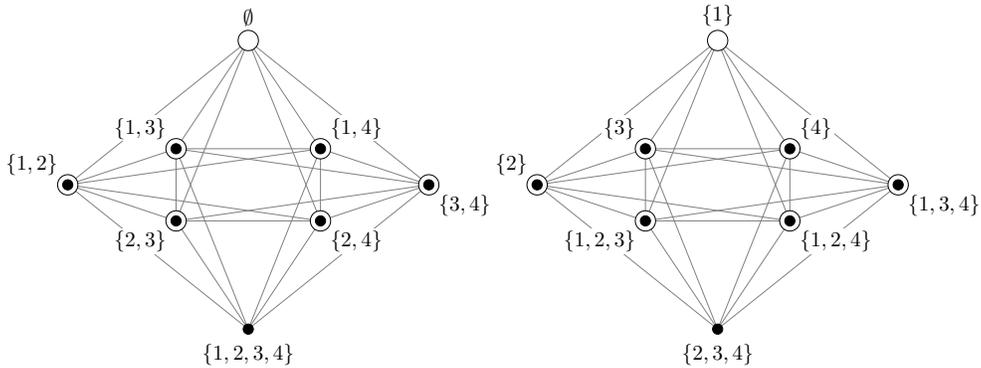}
    \end{center}
    \caption{\label{fig:binary-graph} The Cayley graph $\Cay(\hat{G},\cS)$ for $G = \{-1,1\}^4$ and $\cS = \{S:|S|=0\;\text{or}\;|S|=2\}$.
    The two connected components are $\cT_{\textup{even}}$ (left) and $\cT_{\textup{odd}}$ (right). The vertices 
of $\cT_{\textup{odd}}$ are arranged to correspond to their images in $\cT_{\textup{even}}$ under the graph automorphism $\phi(S) = \{1\}\triangle S$.
We can obtain a chordal cover of $\Cay(\hat{G},\cS)$ by forming maximal cliques on the vertices of $\cT_{\textup{even}}$ marked with filled circles, 
the vertices of $\cT_{\textup{even}}$ marked with open circles, and the images in $\cT_{\textup{odd}}$ of these two cliques under the map $\phi$.}
\end{figure}

To apply the results of Section~\ref{sec:sos} we need to understand the 
graph $\Cay(\hat{G},\cS)$. In the case $n=4$ this graph is shown in Figure~\ref{fig:binary-graph}.
Throughout this section we identify the character $\chi_S\in \hat{G}$
with the subset $S\subseteq [n]$ that indexes it and work exclusively in the 
language of subsets. As such, the vertex set of $\Cay(\hat{G},\cS)$ is
$2^{[n]}$, the collection of all subsets $S\subseteq[n]$. There is an edge between two 
subsets $S,T$ if and only if $|S\triangle T|=2$. This graph is 
often called the \emph{half-cube graph}.

The group operation on characters is multiplication of functions, 
which corresponds to taking the symmetric difference of the subsets 
that index the characters. In other words, if $S,T\subseteq [n]$ then 
\[ \chi_{S}(x)\chi_T(x) = \chi_{S\triangle T}(x)\]
where $S\triangle T = (S\setminus T) \cup (T \setminus S)$. As such,
there is an action of $\hat{G}$ on the vertices of $\Cay(\hat{G},\cS)$
by $S\cdot T = S\triangle T$. Furthermore if $\cT\subseteq 2^{[n]}$ 
is a subset of the vertices of $\Cay(\hat{G},\cS)$ we write 
$S\triangle \cT:= \{S\triangle T: T\in \cT\}$.

We now record some simple observations that follow directly from the  
adjacency relation in $\Cay(\hat{G},\cS)$. For convenience of notation, for $k=0,1,\ldots,n$ let
\[ \cT_k = \{S\subseteq [n]: |S|=k\}.\]
Any edge of $\Cay(\hat{G},\cS)$ either has 
both endpoints in $\cT_k$ for some $k$ or one endpoint in $\cT_k$ 
and the other in $\cT_{k+2}$ for some $k$. Consequently, $\Cay(\hat{G},\cS)$
has two connected components
\[ \cT_{\textup{even}} = \cT_0 \cup \cT_2 \cup \cdots \cup \cT_{2\lfloor n/2\rfloor} \quad\text{and}\quad
    \cT_{\textup{odd}} = \cT_1 \cup \cT_3 \cup \cdots \cup \cT_{2\lceil n/2\rceil-1}.\]
    Define a map $\phi:2^{[n]}\rightarrow 2^{[n]}$ by $\phi(S) = \{1\} \triangle S$. 
Since $|\phi(S)\triangle \phi(T)| = |S\triangle T|$ for all $S,T\in 2^{[n]}$
it follows that $\phi$ extends to an automorphism of $\Cay(\hat{G},\cS)$ that 
exchanges $\cT_{\textup{even}}$ and $\cT_{\textup{odd}}$.

\subsection{Applying Theorem~\ref{thm:mainsos}}
\label{sec:application-binary}
To apply Theorem~\ref{thm:mainsos} from Section~\ref{sec:sos} we need to 
find a subset $\cT\subseteq 2^{[n]}$ of vertices such that 
$\Cay(\hat{G},\cS)$ has a chordal cover with Fourier support $\cT$. 
The following result explicitly describes such a collection of vertices.
\begin{prop}
    \label{prop:chordal-cover-binary}
    The graph $\Cay(\hat{G},\cS)$ has a chordal cover with Fourier support
    \begin{equation}
        \label{eq:binary-freq}
        \cT = \begin{cases}
            \cT_0 \cup \cT_2 \cup \cdots \cup \cT_{\lceil n/2\rceil} & 
            \text{if $\lceil n/2 \rceil$ even}\\
            \cT_1 \cup \cT_3 \cup \cdots \cup \cT_{\lceil n/2\rceil} & 
            \text{if $\lceil n/2\rceil$ odd}.
        \end{cases}
    \end{equation}
\end{prop}
\begin{proof}
    We give a detailed proof in Appendix~\ref{app:binary}. 
\end{proof}
\begin{example}
    To give the flavor of the proof, we discuss the case $n=4$. 
    In this case $\Cay(\hat{G},\cS)$ is shown in Figure~\ref{fig:binary-graph}. 
    Define $\Gamma$ to be the graph with vertex set $2^{[4]}$ and with edges 
    between $S,T\in \cT_{\textup{even}}$ if and only if $||S|-|T||\leq 2$, 
    and edges between $S,T\in \cT_{\textup{odd}}$ if and only if $||\phi(S)|-|\phi(T)|| \leq 2$.
    The graph $\Gamma$ is chordal, with maximal cliques given by
    $\cC_0=\cT_0 \cup \cT_2$, $\cC_2 = \cT_2 \cup \cT_4$, $\phi(\cC_0)$, and $\phi(\cC_2)$. 
    The vertices in cliques $\cC_0$ and $\cC_2$ are indicated by 
    open and filled circles respectively in Figure~\ref{fig:binary-graph}. (The 
    vertices in cliques $\phi(\cC_0)$ and $\phi(\cC_2)$ are similarly marked.)
    If $\cT = \cT_0 \cup \cT_2$ then we can see that $\Gamma$ is a chordal cover of $\Cay(\hat{G},\cS)$
    with Fourier support $\cT$ by observing that $\emptyset \triangle \cC_0 \subseteq \cT$, $\{1,2,3,4\}\triangle \cC_2 \subseteq \cT$, 
    $\phi(\emptyset) \triangle \phi(\cC_0)\subseteq \cT$ and $\phi(\{1,2,3,4\}) \triangle \phi(\cC_2) \subseteq \cT$.
\end{example}

Laurent's conjecture follows directly from
Proposition~\ref{prop:chordal-cover-binary} and Theorem~\ref{thm:mainsos}.
\begin{repthm}{thm:laurent}
    \label{cor:laurent}
    Suppose $f(x) = A_{\emptyset} + \sum_{1\leq i<j\leq n}A_{ij}x_ix_j$ is 
    nonnegative on $G$. Then there is a collection $(h_k)_{k=1}^{|\cT|}$ of functions
    $h_k:G\rightarrow \RR$ each supported on $\cT$ (defined in~\eqref{eq:binary-freq}) such that 
    \[ f(x) = \sum_{k=1}^{|\cT|} h_k(x)^2.\]
    Consequently, any nonnegative quadratic
   form on $G$ is a sum of squares of functions of degree at most $\lceil n/2 \rceil$. 
\end{repthm}
\begin{proof}
    The first assertion follows directly from Proposition~\ref{prop:chordal-cover-binary} 
    and Theorem~\ref{thm:mainsos}. The second assertion
    holds simply because every function supported on $\cT$ has
    degree at most $\lceil n/2\rceil$. 
\end{proof}

The dual version of this result gives a PSD lift of the cut polytope of size $|\cT|$. It follows 
directly from Proposition~\ref{prop:chordal-cover-binary} and Theorem~\ref{thm:moment-main}, and the observation that 
in this case all the characters are real-valued.

\begin{cor}
    \label{cor:cut-polytope}
    The cut polytope $\CUT_n$ has a real PSD lift of size $|\cT|$ 
    given by
    \[ \CUT_n = \left\{\ell\in \RR^{\cS\setminus\emptyset}: \exists y\in \RR^{\cT \triangle \cT} 
    \;\;\text{s.t.}\;\;y_{\emptyset} = 1,\;\; \text{$y_{\{i,j\}} = \ell_{\{i,j\}}$ for $1\leq i<j\leq n$}, 
    \;\;\text{and}\;\;[y_{S\triangle T}]_{S,T\in \cT} \psd 0\right\}\]
    where $\cT$ is defined in~\eqref{eq:binary-freq}.
\end{cor}
In the language used, for example, in~\cite{laurent2003lower} when discussing the Lasserre hierarchy for binary quadratic optimization, 
Corollary~\ref{cor:cut-polytope} could be expressed simply as $Q_{\lceil n/2\rceil} = \CUT_n$
for all $n$.

% ------
\section{Application 2: The case $G=\ZZ_N$ and cyclic polytopes}
\label{sec:cyclic}
% !TEX root = chordal_completion_paper.tex

In this section we apply the results of Section \ref{sec:general} to the case where $G=\ZZ_N$ is the (additive) group of integers modulo $N$. As we will see, this will allow us to obtain a positive semidefinite lift of size $O(d \log (N/d))$ for the regular trigonometric cyclic polytope with $N$ vertices of degree $d$, when $d$ divides $N$.

Recall from Section \ref{sec:prelim}, that the characters of $\ZZ_N$ are indexed by $k \in \ZZ_N$ and are given by:
\[ \chi_k(x) = e^{2i\pi kx/N} \quad \forall x \in \ZZ_N. \]
Thus the Fourier decomposition of a function $f:\ZZ_N\rightarrow \CC$ is given by:
\[ f(x) = \sum_{k \in \ZZ_N} \hat{f}(k) e^{2i\pi kx}. \]
Furthermore, we say that a function $f$ has degree $d$ if it is supported on $\{-d,-(d-1),\dots,d-1,d\}$.
%A function is called \emph{linear} if it is supported on $\{-1,0,1\}$. In general we say that $f$ has degree $d$ if it is supported on $\{-d,-d+1,\dots,0,\dots,d-1,d\}$.

\subsection{The case $\cS = \{-1,0,1\}$: the cycle graph}

In this section we are interested in obtaining sparse sum-of-squares certificates for functions of degree 1 on $\ZZ_N$, i.e., functions supported on $\cS = \{-1,0,1\}$. Note that the real moment polytope $\cM^{\RR}(\ZZ_N,\{-1,1\})$ in this case is just the regular $N$-gon in the plane:
\[
\cM^{\RR}(\ZZ_N,\{-1,1\}) = \conv \Bigl\{ (\cos(2\pi x/N), \sin(2\pi x/N)) : x \in \ZZ_N \Bigr\}.
\]
%Under the invertible $\RR$-linear map $(z,\bar{z})\mapsto (\Re(z),\Im(z))$, we see that $\cM(\ZZ_N,\{-1,0,1\})$ is linearly isomorphic to the regular $N$-gon in $\RR^2$.
% Dually, we are interested in finding an efficient semidefinite description of the moment polytope $\cM(\ZZ_N,\{-1,0,1\})$ which in fact is nothing but the regular $N$-gon in the plane.

To obtain sparse sum-of-squares for nonnegative functions of degree 1 we are going to study the Cayley graph $\Cay(\hat{\ZZ_N},\{-1,0,1\})$. Note that this is simply the cycle graph on $N$ vertices, which we will denote by $C_N$ for simplicity. The object of this section is to show that this graph admits a chordal cover with a small number of frequencies.
We show:

\begin{thm}
\label{thm:maincycle}
Let $N$ be a positive integer greater than 2. Then $C_N$ has a chordal cover with frequencies $\cT \subseteq \hat{\ZZ_N}$ where $|\cT| \leq 3 \log_2 N$. More precisely the set $\cT$ can be described explicitly as follows: Let $k_1 < k_2 < \dots < k_l$ be the positions of the nonzero digits in the binary expansion of $N$ so that $N = \sum_{j=1}^l 2^{k_j}$.  Let $k$ be the largest integer such that $2^k < N$ (i.e., $k = k_1 - 1$ if $N$ is a power of two and $k=k_l$ otherwise). Then the set $\cT$ is given by
\begin{equation}
\label{eq:freqsos}
\cT = \{0\} \cup \{-2^i, 2^i, i=0,\dots,k-1\} \cup \left\{\sum_{j=1}^i 2^{k_j}, i=1,\dots,l-2\right\}.
\end{equation}
\end{thm}
\begin{proof}
The chordal cover is constructed by induction on $N$, see Appendix \ref{app:cylic} for the details. Figure \ref{fig:triangulation_cycle_8_16} shows the chordal cover for $N=8$ and $N=16$.
\end{proof}

\begin{figure}[ht]
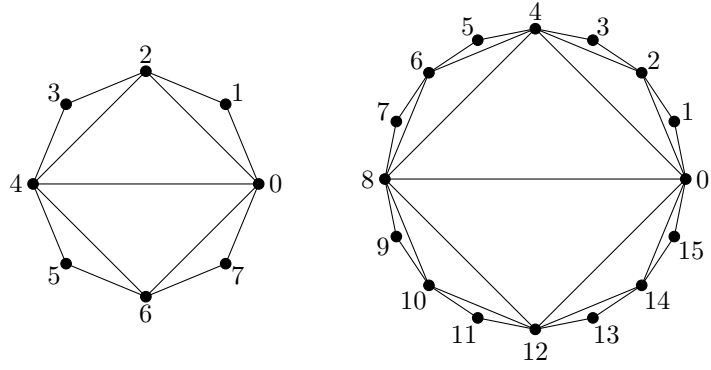

 \centering
 \begin{minipage}{\textwidth}
\centering
 \raisebox{-0.5\height}{\triangulationpoweroftwo{8}{2}{1.5}}
 \qquad
  \raisebox{-0.5\height}{\triangulationpoweroftwo{16}{3}{2}}
 \end{minipage}
 \caption{Triangulation of the 8-cycle with frequencies $\cT = \{-2,-1,0,1,2\}$ and of the 16-cycle with frequencies $\cT=\{-4,-2,-1,0,1,2,4\}$.}
 \label{fig:triangulation_cycle_8_16}
\end{figure}

If we combine the previous theorem with Theorem \ref{thm:mainsos}, we get that any nonnegative degree-1 function on $\ZZ_N$ has a sum-of-squares certificate supported on $\cT$ where $|\cT| \leq 3\log N$. Note that this corresponds to Theorem \ref{thm:intro-ZN-degd} from the introduction for the case $d=1$. Dually, this allows us to obtain a Hermitian positive semidefinite lift of the regular $N$-gon of size $|\cT| \leq 3 \log N$.

In a previous paper \cite{polygonsequivariant} we showed that the $N=2^n$-gon admits a positive semidefinite lift of size $2n-1$. In fact we showed that any linear function on $\ZZ_N$ that is nonnegative can be written as a sum-of-squares of functions supported on $\{0\} \cup \{\pm 2^i, i=0,\dots,n-2\}$. Note that this is the same set of frequencies that we get if we plug $N = 2^n$ in \eqref{eq:freqsos}. Thus Theorem \ref{thm:maincycle} generalizes the result of \cite{polygonsequivariant} for arbitrary $N$.

\subsection{Degree $d$ functions: powers of cycle graph}

In this section we are interested in functions of degree $d$ on $\ZZ_N$ where $d$ divides $N$. We show how to construct a chordal cover of the associated Cayley graph $\Cay(\hat{\ZZ_N},\cS)$ using the chordal cover of the cycle graph constructed in the previous section. This allows us to show that any nonnegative function on $\ZZ_N$ of degree $d$ has a sum-of-squares certificate supported on at most $3 d\log(N/d)$ frequencies.

\subsubsection{Triangulating the Cayley graph}

Observe that the Cayley graph $\Cay(\hat{\ZZ_N},\cS)$ when $\cS = \{-d,\dots,d\}$ is the $d$'th \emph{power} of the cycle graph $C_N$. Recall the definition of power of a graph:
\begin{defn}
Let $\Gr=(V,E)$ be a graph. Given $d \in \NN$, the $d$'th power of $\Gr$ is the graph $\Gr^d = (V,E^d)$ where two vertices $u,v \in V$ are connected if there is a path of length $\leq d$ connecting $u$ and $v$ in $\Gr$.
\end{defn}
Following this observation, we will use the symbol $C_N^d$ to denote the Cayley graph $\Cay(\hat{\ZZ_N},\{-d,\dots,d\})$. Figure \ref{fig:example_graph_power}(left) shows the graph $C_N^d$ for $N=8$ and $d=2$.
\begin{figure}[ht]
  \centering
  \includegraphics[scale=1]{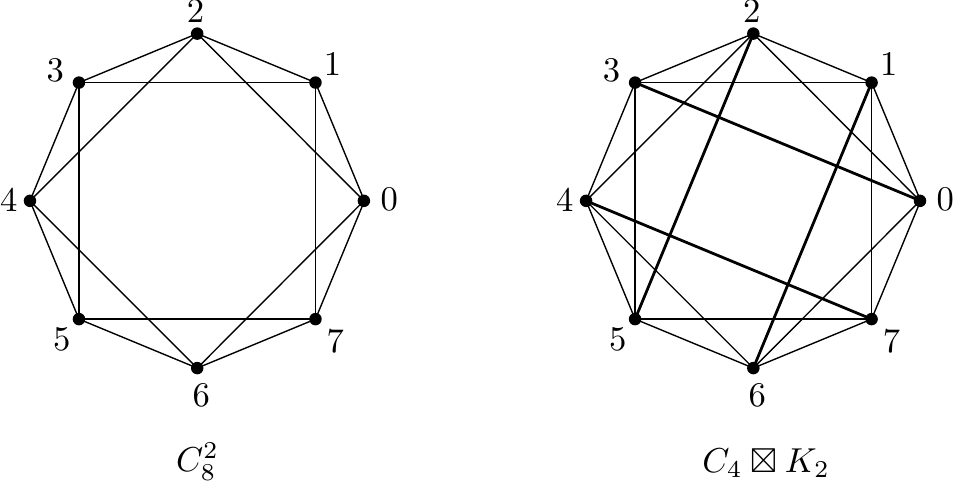}
  \caption{Left: The second power of the cycle graph on 8 nodes: two nodes are connected by an edge if their distance in the cycle graph is at most 2. Right: The graph $C_4 \boxtimes K_2$. Note that $C_8^2 \subset C_4 \boxtimes K_2$. The edges in $C_4 \boxtimes K_2$ that are not in $C_8^2$ are indicated with a heavier line.}
  \label{fig:example_graph_power}
\end{figure}

To construct a triangulation of $C_N^d$ we will actually use the triangulation of the cycle graph $C_N$ constructed in the previous section. For this, we need the following definition of \emph{strong product} of graphs:

\begin{defn}
Given graphs $\Gr=(V,E)$ and $\Gr'=(V',E')$ define the \emph{strong product} of $\Gr$ and $\Gr'$, denoted $\Gr \boxtimes \Gr'$ to be the graph with vertex set $V\times V'$ and where two vertices $(u,u') \in V\times V'$ and $(v,v') \in V\times V'$ are connected if:
\[ (u=v \text{ and } \{u',v'\} \in E') \;\; \text{ or } \;\; (\{u,v\} \in E \text{ and } u'=v') \;\; \text{ or } \;\; (\{u,v\} \in E \text{ and } \{u',v'\} \in E'). \]
\end{defn}
\begin{rem}
An important special case is when one of the graphs, say $\Gr'$, is a complete graph $\Gr' = K_m$. In this case two distinct vertices $(u,u')$ and $(v,v')$ in $\Gr\boxtimes K_m$ are connected if either $u=v$ or $\{u,v\} \in E(\Gr)$.
\end{rem}
Given two graphs $\Gr=(V,E)$ and $\Gr'=(V,E')$ with the same vertex set $V$ we say that $\Gr'$ covers $\Gr$ and we write $\Gr \subseteq \Gr'$ if $E \subseteq E'$.
Our main observation to construct a chordal cover of $C_N^d$ is the following:
\begin{prop}
Let $N$ and $d$ be two integers and assume that $d$ divides $N$. Let $C_N^d$ be the $d$'th power of the cycle graph $C_N$ and let $C_{N/d}$ be the cycle graph on $N/d$ nodes. Then 
\begin{equation}
 \label{eq:Cndproduct}
 C_{N}^d \subseteq C_{N/d} \boxtimes K_d.
\end{equation}
\end{prop}
\begin{proof}
To show the inclusion \eqref{eq:Cndproduct} we first need to identify the vertices of $C_N^d$ with those of $C_{N/d} \boxtimes K_d$. Note that the vertex set of $C_N^d$ can be identified with $\ZZ_N$ and the vertex set of $C_{N/d}$ can be identified with $\ZZ_{N/d}$. We also identify the vertices of $K_d$ with $\{0,\dots,d-1\}$.
By definition of $\boxtimes$, the vertices of $C_{N/d} \boxtimes K_d$ are $\ZZ_{N/d} \times \{0,\dots,d-1\}$. Consider the map:
\begin{equation}
\label{eq:graphidmap}
 \phi: \ZZ_{N/d} \times \{0,\dots,d-1\} \rightarrow \ZZ_N, \quad \phi(q,r) = qd + r.
\end{equation}
This map is well-defined and gives a bijection between $\ZZ_{N/d} \times \{0,\dots,d-1\}$ and $\ZZ_N$.  The map $\phi$ thus identifies vertices of $C_N^d$ with those of $C_{N/d} \boxtimes K_d$.

We now show that, with this identification, inclusion \eqref{eq:Cndproduct} holds. We need to show that if $i,i' \in \ZZ_N$ are connected in $C_N^d$ (i.e., $i-i' \in \{-d,\dots,d\}$) then necessarily $(q,r)$ and $(q',r')$ are connected in $C_{N/d} \boxtimes K_d$ (i.e., $q-q' \in \{-1,0,1\}$), where $(q,r)$ and $(q',r')$ are such that $i=\phi(q,r)$ and $i'=\phi(q',r')$. Consider for $q \in \ZZ_{N/d}$ the set of vertices of $C_N^d$ given by $V_q = \{\phi(q,r) : r =0,\dots,d-1\} \subset \ZZ_N$. Note that $(V_q)_{q \in \ZZ_{N/d}}$ forms a partition of the vertex set of $C_N^d$ and that $|V_q| = d$ for all $q$ (for example if $N=8$ and $d=2$ (Figure \ref{fig:example_graph_power}) $V_0 = \{0,1\}, V_1 = \{2,3\}, V_2 = \{4,5\}, V_4 = \{6,7\}$). It is easy to see that if $i$ and $i'$ are two adjacent vertices of $C_N^d$, then $i$ and $i'$ must be in the same group (i.e., $i,i' \in V_q$) or in adjacent group (i.e., $i \in V_q$ and $i' \in V_{q+1}$ or vice-versa). In other words this means that $q-q' \in \{-1,0,1\}$ which means that $(q,r)$ and $(q',r')$ are connected in $C_{N/d}\boxtimes K_d$.
\end{proof}

%\begin{figure}
%  \centering
%  \includegraphics[scale=1]{figures/example_graph_strong_product.pdf}
%  \caption{Illustration of the strong graph product operation with $K_2$. Each node of the original graph $C_6$ is replaced by two nodes and each edge by a four-clique. Note that $C_{12}^2 \subset C_6 \boxtimes K_2$.}
%   \label{fig:example_graph_strong_product}
%\end{figure}

The previous proposition gives a natural way to construct a chordal cover of $C_{N}^d$ from that of $C_{N/d}$. Indeed if $\Gamma$ is a chordal cover of $C_{N/d}$ then one can show that $\Gamma \boxtimes K_d$ is a chordal cover of $C_{N}^d$ and one can also characterize the maximal cliques of $\Gamma \boxtimes K_d$ in terms of those of $\Gamma$. This is the object of the next proposition.

\begin{prop}
\label{prop:chordalproducts}
Let $\Gr = (V,E)$ be a graph and $d$ be any integer.
\begin{enumerate}
\item If $\Gr'$ is such that $\Gr \subseteq \Gr'$ then $\Gr\boxtimes K_d \subseteq \Gr' \boxtimes K_d$.
\item If $\Gr$ is chordal then $\Gr \boxtimes K_d$ is chordal.
\item All the maximal cliques of $\Gr \boxtimes K_d$ have the form $\cC \times K_d$ where $\cC$ is a maximal clique of $\Gr$.
\end{enumerate}
\end{prop}
\begin{proof}
\begin{enumerate}
\item The first point is clear from the definition of $\boxtimes$.
\item Let $(u_1,v_1)\dots (u_l,v_l)$ be a cycle in $\Gr\boxtimes K_d$ of length $l \geq 4$ where $(u_l,v_l) = (u_1,v_1)$. If there exists $i \in \{1,\dots,l-1\}$ such that $u_i = u_{i+1}$ then the edge $\{(u_i,v_i),(u_{i+2},v_{i+2})\}$ is a chord of the cycle. Otherwise note that $u_1\dots u_l$ is a cycle in $\Gr$ of length $\geq 4$. Since $\Gr$ is chordal there is $1 \leq i,j \leq l-1$ with $j-i \geq 2$ such that $\{u_i,u_j\} \in E$. In this case the edge $\{(u_i,v_i),(u_j,v_j)\}$ is a chord of the cycle.
\item The third property easily follows from the fact that if $\cC=\{(u_i,v_i), i=1,\dots,k\}$ is a clique in $\Gr \boxtimes K_d$ then $\{u_i, i=1,\dots,k\} \subseteq V$ is a clique in $\Gr$.
\end{enumerate}
\end{proof}
%\begin{rem}
%Note that point 2 is not true in general if we replace $K_d$ by another chordal graph $H$. In other words it is possible to have $G$ and $H$ chordal and yet $G\boxtimes H$ not chordal. Consider the graphs $G$ and $H$ shown in Figure \ref{fig:counterexample_chordal_graphs_strong_product} which are both chordal. Note that $G\boxtimes H$ is not chordal because $(1,a)\text{--}(2,b)\text{--}(3,c)\text{--}(4,d)\text{--}(1,a)$ is a chordless cycle of length 4.
%\begin{figure}
%  \centering
%  \includegraphics[scale=1]{figures/counterexample_chordal_graphs_strong_product.pdf}
%  \caption{$G$ and $H$ are two chordal graphs and yet $G\boxtimes H$ is not chordal. Indeed $(1,a)\text{--}(2,b)\text{--}(3,c)\text{--}(4,d)\text{--}(1,a)$ is a chordless cycle of length 4 in $G\boxtimes H$.}
%  \label{fig:counterexample_chordal_graphs_strong_product}
%\end{figure}
%\end{rem}

We can now use the triangulation of the cycle graph constructed in the previous section to obtain a triangulation of $C_N^d$.

\begin{prop}
\label{prop:freqCNd}
Let $N$ and $d$ be two integers and assume that $d$ divides $N$.
If $C_{N/d}$ has a triangulation with frequencies $\cT \subseteq \ZZ_{N/d}$, then $C_N^d$ has a triangulation with frequencies
\begin{equation}
 \label{eq:T'}
 \cT' = \{ dk+r : k \in \cT, r \in \{0,\dots,d-1\} \}.
\end{equation}
Note that $|\cT'| \leq d \cdot |\cT|$.
\end{prop}
\begin{proof}
Let $\Gamma$ be a triangulation $C_{N/d}$ with frequencies $\cT$. By definition, this means that for any maximal clique $\cC$ of $C_{N/d}$, there is $k_{\cC} \in \ZZ_{N/d}$ such that $k_{\cC} + \cC \subseteq \cT$.

By Proposition \ref{prop:chordalproducts}, we know that $\Gamma \boxtimes K_d$ is a chordal cover of $C_N^d$. Let $\cC'$ be a maximal clique of $\Gamma\boxtimes K_d$. By Proposition \ref{prop:chordalproducts}, we know that there exists $\cC$ maximal clique of $\Gamma$ such that $\cC' = \cC \times K_d = \{dq+r : q \in \cC, r \in \{0,\dots,d-1\}\}$. Define $k_{\cC'} = d k_{\cC} \in \ZZ_N$ and note that:
\[ \begin{aligned}
k_{\cC'}+\cC' &= \{d k_{\cC} + dq + r : q \in \cC, r \in \{0,\dots,d-1\}\}\\
&= \{(k_{\cC}+q)d+r: q \in \cC, r \in  \{0,\dots,d-1\}\} \subseteq \cT',
\end{aligned} \]
where the last inclusion follows from the fact that $k_{\cC}+q \in \cT$ whenever $q \in \cC$. We have thus shown that for any maximal clique $\cC'$ of $\Gamma\boxtimes K_d$, there is $k_{\cC'} \in \ZZ_{N}$ such that $k_{\cC'} + \cC' \subseteq \cT'$. Thus this shows that $\Gamma \boxtimes K_d$ is a chordal cover of $C_{N}^d$ with frequencies $\cT'$.
\end{proof}
%\begin{rem}
%The set of frequencies $\cT'$ of Equation \eqref{eq:T'} is not necessarily symmetric, even if $\cT$ is.
%\end{rem}

Combining Proposition \ref{prop:freqCNd} and the triangulation of the cycle graph from Theorem \ref{thm:maincycle} we get the following corollary:
\begin{cor}
\label{cor:triangulationCNd}
Let $N$ and $d$ be two integers and assume that $d$ divides $N$. Then the graph $C_N^d$ has a triangulation with frequencies $\cT \subset \hat{\ZZ_N}$ where $|\cT| \leq 3 d \log (N/d)$.
\end{cor}
Using Theorem \ref{thm:intro-main-sos}, this proves Theorem \ref{thm:intro-ZN-degd} from the introduction concerning nonnegative functions on $\ZZ_N$ of degree $d$.
\begin{repthm}{thm:intro-ZN-degd}
Let $N$ and $d$ be two integers and assume that $d$ divides $N$. Then there exists $\cT \subseteq \ZZ_N$ with $|\cT| \leq 3 d \log (N/d)$ such that any nonnegative function on $\ZZ_N$ of degree at most $d$ has a sum-of-squares certificate supported on $\cT$.
\end{repthm}

Figure \ref{fig:triangulation_cycle_second_power} shows the triangulation of $C_{16}^2$ obtained by triangulating $C_8$ using Theorem \ref{thm:maincycle} and applying the strong graph product with $K_2$.
\begin{figure}[ht]
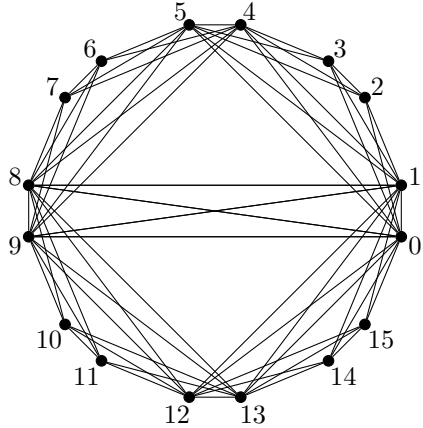

  \centering
  \triangulationpoweroftwoCyclePowerTwo{16}{2}{2.5}
  \caption{Triangulation of the graph $C_{16}^2$ obtained as the strong graph product of $\Gamma$ and $K_2$, where $\Gamma$ is the triangulation of $C_8$ obtained from Theorem \ref{thm:maincycle} and illustrated in Figure \ref{fig:triangulation_cycle_8_16}(left).}
  \label{fig:triangulation_cycle_second_power}
\end{figure}

\subsubsection{Cyclic polytopes}

Observe that the real moment polytope for $G=\ZZ_N$ and $\cS =  \{-d,\dots,d\}$ is given by:
\[
\cM^{\RR}(\ZZ_N,\{-d,\dots,d\}) = \conv \Bigl\{ (\cos(2\pi x/N),\sin(2\pi x/N),\dots,\cos(2\pi dx/N),\sin(2\pi dx/N)), x \in \ZZ_N \Bigr\} \subset \RR^{2d}.
\]
This is just the regular trigonometric cyclic polytope which we abbreviate by $TC(N,2d)$:
\begin{equation}
\label{eq:TC}
 TC(N,2d) = \conv \Bigl\{ M(2\pi x/N) : x=0,1,\dots,N-1 \Bigr\} \subset \RR^{2d},
\end{equation}
where $M(\theta)$ is the degree $d$ trigonometric moment curve:
\[ M(\theta) = \Bigl( \cos (\theta), \sin (\theta), \cos(2\theta), \sin(2\theta), \dots, \cos(d\theta), \sin(d\theta) \Bigr).
\]
Cyclic polytopes play an important role in polyhedral combinatorics \cite{ziegler1995lectures} and satisfy many interesting properties. For example the celebrated \emph{Upper Bound Theorem}, states that for any $2d$ dimensional polytope $P$ with $N$ vertices, $f_i(P) \leq f_i(TC(N,2d))$ for any $i=0,\dots,2d$, where $f_i(P)$ is the number of $i$-dimensional faces of a polytope $P$ \cite{ziegler1995lectures}. Another important property of cyclic polytopes is that they are \emph{neighborly} \cite{gale1963neighborly} (recall that a $2d$-dimensional polytope $P$ is called neighborly if any collection of $d$ vertices of $P$ span a face of $P$). 

The results from this section allow us to obtain a positive semidefinite lift of $TC(N,2d)$ of size $O(d \log (N/d))$ when $d$ divides $N$. More precisely, if we combine Corollary \ref{cor:triangulationCNd} and Theorem \ref{thm:moment-main} we get that $TC(N,2d)$ has a Hermitian positive semidefinite lift of size at most $3d \log (N/d)$, proving Theorem
\ref{thm:intro-TClift} from the introduction:
\begin{repthm}{thm:intro-TClift}
Let $N$ and $d$ be two integers and assume that $d$ divides $N$. Then the trigonometric cyclic polytope $TC(N,2d)$ has a Hermitian positive semidefinite lift of size at most $3d\log (N/d)$.
\end{repthm}

\paragraph{Real positive semidefinite lifts} Using the results of Section \ref{sec:real} one can convert the Hermitian positive semidefinite lift of $TC(N,2d)$ into a real positive semidefinite lift of size at most $4d \log (N/d)$. Indeed, first note that in the case $d=1$, if $\cT$ is the set of frequencies \eqref{eq:freqsos} for the cycle graph $C_{N}$ then $\cT \cup (-\cT)$ has cardinality at most $4 \log (N)$ and the set $\cT \cup (-\cT)$ clearly has an equalizing involution since it is symmetric. Thus this shows that in the case $d=1$, the moment polytope $\cM^{\RR}(\ZZ_N,\{-1,1\})$ has a \emph{real} positive semidefinite lift of size at most $4\log(N)$. For the case where $d > 1$ and $d$ divides $N$ it is not difficult to show that if $\cT$ has an equalizing involution $\sigma$ then $\cT'$ defined in \eqref{eq:T'} also has an equalizing involution. Indeed given $k \in \cT$ and $r \in \{0,\dots,d-1\}$, define $\sigma':\cT'\rightarrow \cT'$ by $\sigma'(dk+r) = d \sigma(k) + d-r-1$ (such a map is well-defined). Then $\sigma'(dk+r) + dk+r = d(\sigma(k) + k) + d$ which is a constant independent of $k$ and $r$, and thus $\sigma'$ is an equalizing involution for $\cT'$. Thus using the symmetric set of frequencies for the cycle graph $C_{N/d}$ (of size $4\log (N/d)$) we get that $\cT'$ has size at most $4d \log (N/d)$ and has an equalizing involution. Thus this shows that $\cM^{\RR}(\ZZ_N,\{-d,\dots,d\})$ has a \emph{real} positive semidefinite lift of size at most $4d \log(N/d)$.

\paragraph{Comparison with LP lifts} One can show that in the regime $N = \Theta(d^2)$ our positive semidefinite lift for $TC(N,2d)$ is provably smaller than any linear programming lift of $TC(N,2d)$. Indeed, the following lower bound on the LP extension complexity of $k$-neighborly polytopes was proved in \cite{fiorini2013combinatorial}:
\begin{prop}(\cite[Proposition 5.16]{fiorini2013combinatorial})
If $P$ be a $k$-neighborly polytope with $N$ vertices then $\xcLP(P) \geq \min(N,(k+1)(k+2)/2)$.
\end{prop}
Since $TC(N,2d)$ is $d$-neighborly, if we choose for example $N = d^2$ then the previous proposition asserts that $\xcLP(TC(d^2,2d)) \geq \Omega(d^2)$ whereas in this case our positive semidefinite has size $O(d\log d)$. This allows us to prove the following result giving a gap between SDP extension complexity and LP extension complexity.
\begin{repcor}{cor:LPSDPgap}
There exists a family $(P_d)_{d \in \NN}$ of polytopes where $P_d \subset \RR^{2d}$ such that 
\[ \frac{\xcPSD(P_d)}{\xcLP(P_d)} = O\left( \frac{\log d}{d} \right). \]
\end{repcor}
The only nontrivial LP lift for cyclic polytopes that we are aware of is a recent construction by Bogomolov et al. \cite{bogomolov2014small} for the cylic polytope
\[ C(N,d) = \conv\Bigl\{ (i,i^2,\dots,i^d) : i=1,\dots,N\Bigr\} \]
of size $(\log N)^{\lfloor d/2 \rfloor}$. Note that this lift has smaller size than the ``trivial'' vertex lift of $C(N,d)$ only when $d < O((\log N) / (\log \log N))$. Their construction for $d=2$ is based on the \emph{reflection relations} framework of Kaibel and Pashkovich \cite{kaibel2011constructing} and the case of general $d$ is then obtained via tensor products.

% ------
\section{Conclusion}
\label{sec:conclusion}
% !TEX root = chordal_completion_paper.tex

In this paper we studied nonnegative functions $f$ defined on a finite abelian
group $G$. We looked at functions $f$ that have a \emph{sparse} Fourier support
$\cS$ and we identified a certain combinatorial condition involving chordal
covers of the Cayley graph $\Cay(\hat{G},\cS)$, that guarantees the existence
of a \emph{sparse} sum-of-squares certificate for any nonnegative function
supported on $\cS$.  We applied our general framework to two special cases.
First we looked at quadratic functions defined on the hypercube $G =
\{-1,1\}^n$ and we showed that any nonnegative quadratic function on $G$ has a
sum-of-squares certificate of degree at most $\lceil n/2 \rceil$. This proves a
conjecture by Laurent from 2003 \cite{laurent2003lower} and shows that the
Lasserre hierarchy for the cut polytope converges after $\lceil n/2 \rceil$
steps.  Second, we looked at nonnegative functions defined on $G=\ZZ_N$, the
group of integers modulo $N$. We showed that when $d$ divides $N$, any degree
$d$ nonnegative function on $\ZZ_N$ has a sum-of-squares certificate with
functions supported on some $\cT$ where $|\cT| \leq O(d\log (N/d))$. From the
geometric point of view, this establishes that the \emph{regular trigonometric
cyclic polytope} of degree $d$ in $N$ vertices has a positive semidefinite lift
of size $O(d \log (N/d))$. For the regime $N=\Theta(d^2)$ this gives us a
family of polytopes in increasing dimensions, where the ratio of the PSD
extension complexity to the LP extension complexity is $O((\log d) / d)$.

\paragraph{Equivariance of lifts}
%  Given a polytope $P$ invariant under the action of some group $G$, we defined in \cite{fawzi2014equivariant} a notion of \emph{equivariant positive semidefinite lift} for $P$  which is a positive semidefinite lift that respects the $G$-symmetry of $P$.
Observe that the moment polytopes $\cM(G,\cS) \subset \CC^{\cS}$ considered in
this paper are invariant under the natural action of $G$ on $\CC^{\cS}$ given
by $\rho(x) = \diag([\chi(x)]_{\chi \in \cS})$. One can show that the Hermitian
PSD lifts for $\cM(G,\cS)$ considered in this paper respect this symmetry of
$\cM(G,\cS)$, in a certain formal sense known as \emph{equivariant lifts}
defined in \cite{fawzi2014equivariant}.
% we introduced the notion of \emph{equivariant positive semidefinite lift} which is a positive semidefinite lift that respects the symmetry of a polytope $P$. In our case, one can show that the lift described in Theorem \ref{thm:intro-main-moments} for $\cM(G,\cS)$ is \emph{equivariant} with respect to the natural action of $G$ on $\CC^{\cS}$.

%\cbstart
\paragraph{Open problems}
We conclude by briefly discussing two concrete problems about regular trigonometric 
cyclic polytopes arising from this work.  
In previous work we showed that any $\ZZ_N$-equivariant PSD lift of the regular $N$-gon must 
have size at least $\ln(N/2)$~\cite[Theorem 9]{polygonsequivariant}. Since all the Hermitian 
PSD lifts in this paper are equivariant, this shows that it 
is not possible to construct substantially better equivariant PSD lifts of the regular 
$N$-gons than the construction given in Theorem~\ref{thm:intro-TClift} with $d=1$.
It would be interesting to establish corresponding lower bounds for general $d$.
\begin{problem}
	Find lower bounds on the size of ($\ZZ_N$-equivariant) PSD lifts of the regular 
	trigonometric cyclic polytope of degree $d$ in $N$ vertices.
\end{problem}
Problem 9.9 of~\cite{psdranksurvey} asks for a ``family of polytopes that exhibits a
large (e.g.\ exponential) gap between its nonnegative and psd ranks''. 
We have found that the regular trigonometric cyclic polytope of degree $d$ in $N=d^2$ vertices 
gives an explicit family of polytopes for which we can prove a significant (but far from exponential) 
gap between PSD and LP extension complexity. It may be possible to prove a better lower bound 
(parameterized by $N$ and $d$) on the LP extension complexity for this family of polytopes since the bound we use only 
uses the fact that the polytope is neighborly. Such a lower bound may allow us to 
establish a larger gap between the LP and PSD extension complexity of these polytopes.
\begin{problem}
	Find $N(d)$ to make the gap between the LP and PSD extension complexity of regular trigonometric 
	cyclic polytopes of degree $d$ in $N(d)$ vertices as large as possible.
\end{problem}
%\cbend

% ------
%\section{Extensions}
%\label{sec:extensions}
%\input{section_extensions.tex}

\appendix
\section{Additional proofs}
\subsection{Proofs from Section~\ref{sec:moment}}
\label{app:moment}
% !TEX root = chordal_completion_paper.tex

In this section we provide detailed proofs of the main results
in Section~\ref{sec:moment}.
\begin{proof}[Proof of Lemma~\ref{lem:moment-obv}]
    Suppose $\mu$ is a probability measure supported on $G$. 
    Then because $\mu$ is a probability measure, $(\EE_{\mu}[\chi])_{\chi\in \cS}$ certainly satisfies
    $\EE_{\mu}[1_{\hat{G}}] = 1$ and $\EE_{\mu}[|f|^2] \geq 0$ whenever $f\in \cF(G,\CC)$. Conversely, suppose
    $\ell$ is a linear functional on $\cF(G,\CC)$ such that $\ell(1_{\hat{G}})=1$ and $\ell(|f|^2) \geq 0$ for all $f\in \cF(G,\CC)$. 
    We show that $\ell(\cdot)$ coincides with $\EE_{\mu}[\cdot]$ for some probability measure $\mu$ supported on $G$.
    For any $x\in G$ define $\mu(\{x\}) = \ell(\delta_x)$. Since $\delta_x = |\delta_x|^2$ we have that $\mu(\{x\})=\ell(|\delta_x|^2)\geq 0$ for all $x\in G$. 
    Since $\ell$ is linear 
    \[ \sum_{x\in G}\mu(\{x\}) = \ell\left(\sum_{x\in G}\delta_x\right) = \ell(1_{\hat{G}}) = 1.\]
    Hence $\mu$ defines a probability measure supported on $G$ and $\ell(\cdot)$ is exactly the corresponding 
    expectation $\EE_\mu[\cdot]$. The second description of $\cM(G,\hat{G})$ in the lemma follows by rewriting the condition 
    $\ell(|f|^2)$ for all $f\in \cF(G,\CC)$ in coordinates with respect to the character basis as
    \[ \ell(\big|\sum_{\chi\in \hat{G}} \hat{f}(\chi)\chi\big|^2) 
        %= \ell(\bar{\left(\sum_{\chi\in \hat{G}} \hat{f}(\chi)\right)}\left(\sum_{\chi'\in \hat{G}}\hat{f}(\chi')\right))
    = \sum_{\chi,\chi'\in \hat{G}}\ell(\bar{\chi}\chi')\bar{\hat{f}(\chi)}\hat{f}(\chi')\quad\text{for all $(\hat{f}(\chi))_{\chi\in \hat{G}}\in \CC^{\hat{G}}$.}\] 
    This is exactly the definition of the matrix $M(\ell) = [\ell_{\bar{\chi}\chi'}]_{\chi,\chi'\in \hat{G}}$ being Hermitian positive semidefinite.
\end{proof}

\begin{proof}[Proof of Proposition~\ref{prop:moment-unstructured}]
    If $(\ell_{\chi})_{\chi\in \cS}\in \cM(G,\cS)$ then from Corollary~\ref{cor:moment-structured} there is $(y_{\chi})_{\chi\in \hat{G}}$
    such that $y_\chi = \ell_\chi$ for all $\chi\in \cS$, $y_{1_{\hat{G}}}=1$, and $M(y) \psd 0$. 
    Hence we can take $Y_{\chi,\chi'} = y_{\bar{\chi}\chi'}$ to show that $(\ell_\chi)_{\chi\in \cS}$ is an element of the right hand side of~\eqref{eq:moment-structured}.

    Conversely, suppose there exists $Y\in \HH_+^{\hat{G}}$ with $Y_{\chi,\chi'} = \ell_{\bar{\chi}\chi'}$ whenever $\bar{\chi}\chi'\in\cS$ 
    and $Y_{\chi,\chi}=1$ for all $\chi\in\hat{G}$. Our task is to construct, from $Y$ some $(y_\chi)_{\chi\in \hat{G}}$ such that $y_{\chi}=\ell_\chi$
    for all $\chi\in\cS$, $y_{1_{\hat{G}}}=1$, and $M(y)\psd 0$. Observe that $\hat{G}$ acts on 
    Hermitian matrices $\HH^{\hat{G}}$ indexed by elements of $\hat{G}$ by simultaneously permuting the rows and columns, i.e.\ by
$[\lambda \cdot Y]_{\chi,\chi'}= Y_{\bar{\lambda}\chi,\bar{\lambda}\chi'}$. 
    We construct a new matrix $Z$ by averaging $Y$ over this group action:
    \[ Z_{\chi,\chi'} := \frac{1}{|\hat{G}|}\sum_{\lambda\in \hat{G}} [\lambda \cdot Y]_{\chi,\chi'} = 
    \frac{1}{|\hat{G}|} \sum_{\lambda\in \hat{G}} Y_{\bar{\lambda}\chi,\bar{\lambda}\chi'}.\]
    Since the action of $\lambda$ is by simultaneously permuting rows and columns each $\lambda \cdot Y$, and hence $Z$ itself, is positive 
    semidefinite with ones on the diagonal. Since we have constructed $Z$ by averaging over a group action, $Z$ is fixed by the action 
    and so satisfies $Z_{\bar{\lambda}\chi,\bar{\lambda}\chi'} = Z_{\chi,\chi'}$
    for all $\chi,\chi'\in \hat{G}$. Consequently there is some $(y_\chi)_{\chi\in \hat{G}}$ such that $Z_{\chi,\chi'} = y_{\bar{\chi}\chi'}$
    for all $\chi,\chi'\in \hat{G}$. It remains to show that if $\bar{\chi}\chi'\in \cS$ then $y_{\bar{\chi}\chi'} := Z_{\chi,\chi'} = Y_{\chi,\chi'} := \ell_{\bar{\chi}\chi'}$.
    This holds because if $\bar{\chi}\chi'\in \cS$ then 
    \[ Z_{\chi,\chi'} = \frac{1}{|\hat{G}|} \sum_{\lambda\in \hat{G}} Y_{\bar{\lambda}\chi,\bar{\lambda}\chi'} = 
    \frac{1}{|\hat{G}|}\sum_{\lambda\in \hat{G}}\ell_{\bar{\chi}\chi'} = \ell_{\bar{\chi}\chi'}.\]
    Hence $y$ has all the desired properties, completing the proof.
\end{proof}

\begin{proof}[Proof of Theorem~\ref{thm:moment-main}]
    First we show that $\cM(G,\cS)$ is a subset of the right-hand-side of~\eqref{eq:moment-main}. To see this observe 
    that the right-hand-side of~\eqref{eq:moment-structured} is certainly contained in the right-hand-size of~\eqref{eq:moment-main}.

    We now establish the reverse inclusion. Suppose $(\ell_\chi)_{\chi\in \cS}$ is such that there exists $(y_\chi)_{\chi\in \cT^{-1}\cT}$
    with $y_\chi = \ell_\chi$ for all $\chi\in \cS$, $y_{1_{\hat{G}}}=1$ and $M_{\cT}(y)\psd 0$. 
    Let $\Gamma$ be a chordal cover of $\Cay(\hat{G},\cS)$ with Fourier support $\cT$. Specifically $\Gamma$ has the property that 
    for every maximal clique $\cC$ of $\Gamma$ there is $\chi_{\cC}\in \hat{G}$ such that $\chi_{\cC}\cC \subseteq \cT$.     
    
    Define the $\Gamma$-partial matrix
    $Y_{\eta,\eta'} = y_{\bar{\eta}\eta'}$ whenever $(\eta,\eta')$ form an edge of $\Gamma$ and $Y_{\chi,\chi}=1$ for all $\chi\in \hat{G}$. 
    This is well defined because
    if $(\eta,\eta')$ is an edge of $\Gamma$, then $\bar{\eta}\eta'\in \cT^{-1}\cT$. 
    To see this observe that any edge of $\Gamma$ is contained in a maximal clique $\cC$, and so there is some 
    $\chi_{\cC}\in \hat{G}$ such that $\chi_{\cC}\eta\in \cT$ and $\chi_{\cC}\eta'\in \cT$. 
    Consequently $\bar{\eta}\eta' = \bar{\chi_{\cC}\eta}\chi_{\cC}\eta' \in \cT^{-1}\cT$. 
       
    We show that $Y[\cC,\cC] \psd 0$ for all maximal cliques $\cC$ of the chordal graph $\Gamma$. This holds because
    \[ Y[\cC,\cC] = [y_{\bar{\eta}\eta'}]_{\eta,\eta'\in \cC} =
    [y_{\bar{\chi_{\cC}\eta}\chi_{\cC}\eta'}]_{\eta,\eta'\in \cC} = [y_{\bar{\chi}\chi'}]_{\chi,\chi'\in \chi_{\cC}\cC}\]
    which, since $\chi_{\cC}\cC \subseteq \cT$, is a principal submatrix of the positive semidefinite (by assumption) matrix $M_{\cT}(y) = [y_{\bar{\chi}\chi'}]_{\chi,\chi'\in \cT}$.
    By the chordal completion theorem (Theorem~\ref{thm:completion}) we can complete $Y$ to a positive semidefinite matrix $Y\in \S_+^{\hat{G}}$.
        The completed matrix has unit diagonal and, whenever $\bar{\chi}\chi'\in \cS$, 
        \[ Y_{\chi,\chi'} = y_{\bar{\chi}\chi'} = \ell_{\bar{\chi}\chi'}\]
        where the first equality is because the edge set of $\Gamma$ contains the edge set of $\Cay(\hat{G},\cS)$
        and the second is from the definition of $y$. Hence, 
        by Proposition~\ref{prop:moment-unstructured}, $(\ell_{\chi})_{\chi\in \cS}\in \cM(G,\cS)$, as we require.
\end{proof}

\subsection{Proof of Lemma~\ref{lem:real}}
\label{app:real}
% !TEX root = chordal_completion_paper.tex

\begin{proof}[Proof of Lemma~\ref{lem:real}]
First note that $L \in \HH_+^d$ if and only if $\bar{L}\in  \HH_+^d$
    which holds if and only if the block diagonal matrix $\left[\begin{smallmatrix} L & 0\\0 & \bar{L}\end{smallmatrix}\right] \in \HH^{2d}_+$.
    Conjugating by a unitary matrix we obtain 
    \begin{equation}
        \label{eq:goemans}
        \begin{bmatrix} \frac{1}{\sqrt{2}}I & \frac{1}{\sqrt{2}}I\\
\frac{i}{\sqrt{2}}I & -\frac{i}{\sqrt{2}}I\end{bmatrix}
\begin{bmatrix} L & 0\\0 &
\bar{L}\end{bmatrix}\begin{bmatrix}\frac{1}{\sqrt{2}}I & \frac{1}{\sqrt{2}}I\\
\frac{i}{\sqrt{2}}I & -\frac{i}{\sqrt{2}}I\end{bmatrix}^*
    = \begin{bmatrix} \phantom{-}\Re[L] & \Im[L]\\-\Im[L] & \Re[L]\end{bmatrix}.
    \end{equation}
    % = \begin{bmatrix} \frac{L+JLJ}{2} & \frac{L-JLJ}{2i}\\
% \frac{JLJ-L}{2i} & \frac{L+JLJ}{2}\end{bmatrix}.\]
We have simply recovered the familiar realization of $\HH_+^d$ 
as a section of $\S_+^{2d}$, and have not yet used any special 
properties of $\cL$. To complete the proof it remains to carefully choose
a $2d\times 2d$ orthogonal matrix $Q$ (depending on $J$) such that  
\[ Q \begin{bmatrix}\phantom{-}\Re[L] & \Im[L]\\-\Im[L] & \Re[L]\end{bmatrix}Q^T
= \begin{bmatrix} \Re[L]-J\Im[L] & 0 \\0 & \Re[L] - J\Im[L]\end{bmatrix}\quad\text{for all $L\in \cL$}.\]
Observe that $J^2 = I$ and $J^TJ=I$ imply that $J= J^T$.
Since $JLJ^T = \bar{L}$ we have that for all $L\in \cL$,
\begin{equation}
        \label{eq:real-im}
        \Re[L] = \frac{L+JLJ}{2}\quad\text{and}\quad
        \Im[L] = \frac{L - JLJ}{2i}.
    \end{equation}
    It follows that for all $L\in \cL$,
    $\Re[L]$ and $\Im[L]$ commute and anti-commute respectively with $J$,
    i.e.,
    \begin{equation}
        \label{eq:re-im-sym}
        J\Re[L] = \Re[L]J\quad\text{and}\quad J\Im[L] = -\Im[L]J.
    \end{equation}
Choosing $Q$ to be the orthogonal matrix $Q = \frac{1}{\sqrt{2}}\left[\begin{smallmatrix} I & J\\-J & I\end{smallmatrix}\right]$ we obtain 
\begin{align*}
    \begin{bmatrix} \frac{1}{\sqrt{2}}I & \frac{1}{\sqrt{2}}J\\
-\frac{1}{\sqrt{2}}J & \frac{1}{\sqrt{2}}I\end{bmatrix}
\begin{bmatrix}\phantom{-}\Re[L] & \Im[L]\\-\Im[L] & \Re[L]\end{bmatrix}
    % \frac{L+JLJ}{2} & \frac{L-JLJ}{2i}\\
% \frac{JLJ-L}{2i} & \frac{L+JLJ}{2}\end{bmatrix}
    \begin{bmatrix} \frac{1}{\sqrt{2}}I & \frac{1}{\sqrt{2}}J\\
-\frac{1}{\sqrt{2}}J & \frac{1}{\sqrt{2}}I\end{bmatrix}^T
    % & = \begin{bmatrix} \frac{L+JLJ}{2} + \frac{LJ-JL}{2i} & 0\\
% 0 & \frac{L+JLJ}{2} + \frac{LJ-JL}{2i}\end{bmatrix}\\
  & = \begin{bmatrix} \Re[L] - J\Im[L] & 0\\0 &
\Re[L]-J\Im[L]\end{bmatrix}.\end{align*}
Clearly this last matrix is positive semidefinite if and only if 
the real symmetric matrix $\Re[L]-J\Im[L]$ is positive semidefinite,
completing the proof.
 \end{proof}

\subsection{Proof of Proposition~\ref{prop:chordal-cover-binary}}
\label{app:binary}
% !TEX root = chordal_completion_paper.tex

\begin{proof}[Proof of Proposition~\ref{prop:chordal-cover-binary}]
The proof proceeds as follows. First we define a graph $\Gamma$ and prove that it is 
   a chordal cover of $\Cay(\hat{G},\cS)$. We then characterize the maximal cliques of 
   $\Gamma$. Finally we show that for any maximal clique $\cC$ of $\Gamma$ there is some 
   $S\in 2^{[n]}$ such that $S\triangle \cC \subseteq \cT$, establishing the stated result.
    We consider the two cases $\lceil n/2\rceil$ even and $\lceil n/2\rceil$ odd
    separately. We describe the argument in detail in the case where $\lceil n/2\rceil$ is 
    even, and just sketch the required modifications in the case where $\lceil n/2\rceil$ is odd.
    
    Assume that $\lceil n/2\rceil$ is even. Let $\Gamma$ be the graph with 
    vertex set $2^{[n]}$ such that two vertices $S,T$ are 
    adjacent in $\Gamma$ if and only if either 
    \begin{itemize}
        \item $|S|$ and $|T|$ are both even and $||S|-|T|| \leq 2$ or
        \item $|S|$ and $|T|$ are both odd and $||\phi(S)| - |\phi(T)|| \leq 2$.
    \end{itemize}
    Note that just like $\Cay(\hat{G},\cS)$, the graph $\Gamma$ also has two connected components with vertex sets $\cT_{\textup{even}}$ and $\cT_{\textup{odd}}$.
    Furthermore, $\phi$ (defined in Section~\ref{sec:cayley-binary}) is also an automorphism of $\Gamma$ that exchanges these two connected components.
    Observe that if $|S\triangle T| =2$ (i.e.\ $S$ and $T$ are adjacent in $\Cay(\hat{G},\cS)$)
    then both $||S|-|T|| \leq2$ and $||\phi(S)| - |\phi(T)|| \leq 2$
    hold. Hence if $S$ and $T$ are adjacent in $\Cay(\hat{G},\cS)$ they are also adjacent in $\Gamma$. 

    We now show that $\Gamma$ is a chordal graph. Let the vertices $S_1,S_2,S_3,\ldots,S_k$ form a $k$-cycle (with $k\geq 4$) in $\Gamma$
    such that each of the $S_i\in \cT_{\textup{even}}$. Without loss
    of generality assume that $|S_1| \leq |S_i|$ for $1\leq i \leq k$.
    We show that the cycle $S_1,S_2,S_3,\ldots,S_k$ has a chord. If $|S_2|=|S_1|$ then
    $||S_1|-|S_3|| = ||S_2|-|S_3||\leq 2$ (since $S_2$ and $S_3$ are adjacent)
    and so there is a chord between $S_1$ and $S_3$. Otherwise suppose $|S_2| = |S_1|+2$. 
    Because $S_1$ and $S_k$ are adjacent we see that either $|S_k|=|S_1| = |S_2|-2$  or $|S_k| = |S_1|+2 = |S_2|$ 
    and so there is a chord between $S_2$ and $S_k$. Now suppose $S_1,S_2,S_3,\ldots,S_k$ form a $k$-cycle (with $k\geq 4$) 
    in $\Gamma$ such that each of the $S_i\in \cT_{\textup{odd}}$. Then the image of the cycle under $\phi$ is a $k$-cycle in 
    $\Gamma$ with vertices in $\cT_{\textup{even}}$ and so it has a chord. Since $\phi$ is an automorphism of $\Gamma$
    it follows that $S_1,S_2,S_3,\ldots,S_k$ also has a chord. So $\Gamma$ is a chordal cover of $\Cay(\hat{G},\cS)$. 

    The subgraphs of  $\Gamma$ induced by the vertex sets 
    $\cC_k:= \cT_{k} \cup \cT_{k+2}$ (for $k=0,2,\ldots,2\lfloor n/2\rfloor-2$) and the vertex sets 
    $\phi(\cC_k)$ (for $k=0,2,\ldots,2\lfloor n/2\rfloor-2$) are cliques in $\Gamma$. In fact, these 
    are maximal cliques in $\Gamma$. To show that each $\cC_k$ is a maximal clique, suppose $S$ is a vertex that is not 
    in $\cC_k$. Then either $|S|$ is odd (in which case $S$ is not adjacent to any element of $\cC_k$)
    or $|S| \leq k-2$ (in which case $S$ is not adjacent to any $T\in \cT_{k+2}$) or $|S|\geq k+4$
    (in which case $S$ is not adjacent to any $T\in \cT_{k}$). Hence there is no inclusion-wise larger 
    clique of $\Gamma$ containing $\cC_k$. Since $\phi$ is an automorphism of $\Gamma$ it follows that the $\phi(\cC_k)$
    are also maximal cliques of $\Gamma$. Finally, there are no other maximal cliques in $\Gamma$ because
    every edge of $\Gamma$ is contained either in $\cC_k$ or $\phi(\cC_k)$ for some $k=0,2,\ldots,2\lfloor n/2\rfloor-2$.
    
     It remains to show that for any maximal clique $\cC_k$ (for $k=0,2,\ldots,2\lfloor n/2\rfloor-2$)
     of $\Gamma$ there is $S_k\in 2^{[n]}$ such that $S_k\triangle \cC_k \subseteq \cT$.
     This is sufficient to establish that $\Cay(\hat{G},\cS)$ has a chordal cover with Fourier support $\cT$ 
     because for the cliques $\phi(\cC_k)$ we have that $\phi(S_k)\triangle \phi(\cC_k) = S_k\triangle \cC_k \subseteq \cT$.
     The following gives valid choices of $S_k$ (for $k=0,2,\ldots,2\lfloor n/2\rfloor-2$). 
     \begin{itemize}
         \item If $k\leq \lceil n/2\rceil -2$ then $\cC_k\subseteq \cT$ so we can take $S_k=\emptyset$.
         \item If $k \geq \lceil n/2\rceil$ and $n$ is even then $n=2\lceil n/2\rceil$ and so
             $n-k-2 \leq \lceil n/2\rceil -2$. Hence $[n]\triangle \cC_k = \cC_{n-k-2} \subseteq \cT$ so we can take $S_k = [n]$.
         \item If $k \geq \lceil n/2\rceil$ and $n$ is odd then $n=2\lceil n/2\rceil-1$ and so 
             $n-k+1 \leq \lceil n/2\rceil$. Hence 
             \[ \phi([n])\triangle \cC_k = [n]\triangle \phi(\cC_k) \subseteq [n]\triangle (\cT_{k-1} \cup \cT_{k+1} \cup \cT_{k+3}) 
                 \subseteq \cT_{n-k-3}\cup \cT_{n-k-1} \cup \cT_{n-k+1} \subseteq \cT\] 
            so we can take $S_k = \phi([n])$.
     \end{itemize}
    This completes the argument in the case where $\lceil n/2\rceil$ is even.

   In the case where $\lceil n/2\rceil$ is odd we exchange the roles of the odd and 
   even components in the definition of $\Gamma$ and throughout the argument. More precisely, two vertices $S,T$ are 
    adjacent in $\Gamma$ if and only if either 
    \begin{itemize}
        \item $|S|$ and $|T|$ are both odd and $||S|-|T|| \leq 2$ or
        \item $|S|$ and $|T|$ are both even and $||\phi(S)| - |\phi(T)|| \leq 2$.
    \end{itemize}
    It is still the case that $\Gamma$ is a chordal cover of $\Cay(\hat{G},\cS)$. Its maximal cliques are now
    $\cC_k:= \cT_{k} \cup \cT_{k+2}$ for $k=1,3,\ldots,2\lceil n/2\rceil -3$ together with the $\phi(\cC_k)$. Note that
   the cliques are now indexed by odd integers. As before, we can choose the $S_k$ (for $k=1,3,\ldots,2\lceil n/2\rceil -3$) to be $S_k = \emptyset$ if $k\leq \lceil n/2\rceil-2$,
$S_k = [n]$ if $k\geq \lceil n/2\rceil$ and $n$ is even, and $S_k = \phi([n])$ if $k\geq\lceil n/2\rceil$ and $n$ is odd.

% The final change in the argument is in the construction of the 
%   $S_k$ (for $k=1,3,\ldots,2\lceil n/2\rceil -3$). It is straightforward to check that the following choices of the $S_k$ are valid.
%   \begin{itemize}
%       \item If $k\leq \lceil n/2\rceil -2$ then $\cC_k \subseteq \cT$ so we can take $S_k = \emptyset$.
%       \item If $k \geq \lceil n/2\rceil$ and $n$ is even then  $n=2\lceil n/2\rceil$ and so
%           $n-k-2 \leq \lceil n/2\rceil -2$. Hence $[n]\triangle \cC_k = \cC_{n-k-2}\subseteq \cT$ so we can take $S_k = [n]$.
%       \item If $k \geq \lceil n/2\rceil$ and $n$ is odd then $n=2\lceil n/2\rceil-1$ and so 
%             $n-k+1 \leq \lceil n/2\rceil$. Hence 
%             \[ \phi([n])\triangle \cC_k = [n]\triangle \phi(\cC_k) \subseteq [n]\triangle (\cT_{k-1} \cup \cT_{k+1} \cup \cT_{k+3}) 
%                 \subseteq \cT_{n-k-3}\cup \cT_{n-k-1} \cup \cT_{n-k+1} \subseteq \cT\] 
%            so we can take $S_k = \phi([n])$.
%     \end{itemize}
   This completes the argument in the case where $\lceil n/2\rceil$ is odd.
   \end{proof}

\subsection{Proof of Theorem~\ref{thm:maincycle}: triangulation of the cycle graph}
\label{app:cylic}
% !TEX root = chordal_completion_paper.tex

In this appendix we prove Theorem \ref{thm:maincycle} concerning the triangulation of the cycle graph $C_N$. 
Theorem \ref{thm:triangulationcycle} below shows how to construct a triangulation of the cycle graph $C_{N+1}$ on $N+1$ nodes, by induction. The triangulation of $C_N$ used to obtain Theorem \ref{thm:maincycle} will then be obtained simply by contracting a certain edge of the triangulation of $C_{N+1}$ (more details below). We thus start by describing a triangulation of the $N+1$-cycle.
\begin{thm}[Triangulation of the cycle graph on $N+1$ vertices]
\label{thm:triangulationcycle}
Let $N$ be an integer greater than or equal 2. 
 Let $k_1 < \dots < k_l$ be the position of the nonzero digits in the binary expansion of $N$, i.e., $N = \sum_{i=1}^l 2^{k_i}$. Let $k$ be the largest integer such that $2^k < N$ (i.e., $k = k_1 - 1$ if $N$ is a power of two and $k=k_l$ otherwise). Then there exists a triangulation of the cycle graph on $N+1$ nodes $C_{N+1}$ with frequencies:
\begin{equation}
 \label{eq:freqtriangulation}
 \cT = \{0\} \cup \{ \pm 2^i,i=0,\dots,k\} \cup \left\{\sum_{j=1}^i 2^{k_j},i=1,\dots,l-1\right\}.
\end{equation}
\end{thm}
\begin{proof}
The proof of the theorem is by induction on $N$.  Consider the cycle graph on $N+1$ nodes where nodes are labeled $0,1,\dots,N$. To triangulate the graph, we first put an edge between nodes $0$ and $2^k$ and another edge between nodes $2^k$ and $N$, where $2^k$ is the largest power of two that is strictly smaller than $N$. This is depicted in Figure \ref{fig:recursive_triangulation_log2}.
\begin{figure}[ht]
  \centering
  \includegraphics[scale=1]{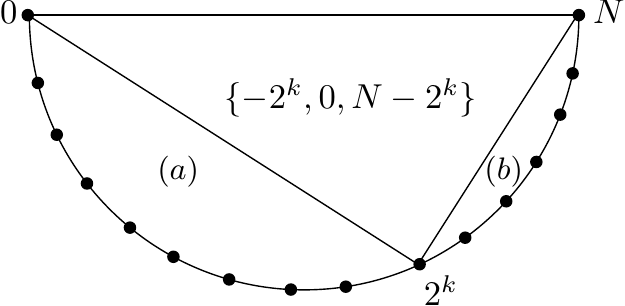}
  \caption{Recursive triangulation of the cycle $0\dots N$ on $N+1$ vertices}
  \label{fig:recursive_triangulation_log2}
\end{figure}

Note that the frequencies used by the triangle $\{0,2^k,N\}$ are equivalent, by translation, to $\{-2^k,0,N-2^k\}$. We now use induction to triangulate the two remaining parts of the cycle (denoted (a) and (b) in Figure \ref{fig:recursive_triangulation_log2}):

$\bullet$ For part (a), which is a cycle graph labeled $0\dots N'$ with $N'=2^k$, the induction hypothesis gives us a triangulation with frequencies 
\begin{equation}
\label{eq:Ka}
 \cT_a = \{0\}\cup\{ \pm 2^i,i=0,\dots,k-1\}.
\end{equation}

$\bullet$ For part (b) of the graph, we use induction on the cycle $2^k\dots N$ which is, by translation, equivalent to the cycle with labels $0\dots N''$ where $N''=N-2^k$. We distinguish two cases:
\begin{itemize}
\item[--] If $N = 2^{k+1}$, then we have $N'' = 2^k$ and induction gives a triangulation of (b) with the same frequencies as for part (a). Thus in this case we get a triangulation of the full $(N+1)$-cycle with frequencies:
%\marginparsmall{Make comment somewhere about ``negative'' frequencies. Here I am doing as if ``$-2^k=2^k$''...}
\[ \cT_a \cup \{-2^k,0,2^k\} = \{0\}\cup\{ \pm 2^i,i=0,\dots,k\} \]
which is what we want.
\item[--] Now assume that $N < 2^{k+1}$, which means that the most significant bit of $N$ is at position $k=k_l$. Thus the binary expansion of $N''=N-2^k$ is the same as that of $N$ except that the bit at position $k=k_l$ is replaced with a 0. Let $k''$ be the largest integer such that $2^{k''} < N''$. Using induction we get a triangulation of the cycle $0\dots N''$ using frequencies where
\begin{equation}
 \label{eq:Kb}
 \cT_b = \{0\} \cup \{\pm 2^i,i=0,\dots,k''\} \cup \left\{\sum_{j=1}^i 2^{k_j}, j=1,\dots,l-2\right\}.
\end{equation}
Combining the triangulation of parts (a) and part (b) we get a triangulation of the $(N+1)$-cycle with frequencies 
\[ \underbrace{\{-2^k,0,N-2^k\}}_{\text{triangle $\{0,2^k,N\}$}} \cup \cT_a \cup \cT_b. \]
%\begin{equation}
%\label{eq:bigunion}
%\cK = \underbrace{\{-2^k,0,N-2^k\}}_{\text{triangle $\{0,2^k,N\}$}}\cup\underbrace{\{0\}\cup\{2^i,i=0,\dots,k-1\}}_{\text{part }(a)} \cup \underbrace{\{0\} \cup \{2^i,i=0,\dots,k''\} \cup \left\{\sum_{j=1}^i 2^{k_j}, j=1,\dots,l-2\right\}}_{\text{part }(b)}
%\end{equation}
Given the expressions \eqref{eq:Ka} and \eqref{eq:Kb} for $\cK_a$ and $\cK_b$, and noting that $k'' \leq k-1$ and that $N - 2^k = \sum_{j=1}^{l-1} 2^{k_j}$, one can check that the triangulation has frequencies in
\[
 \cT = \{0\} \cup \{\pm 2^i,i=0,\dots,k\} \cup \left\{\sum_{j=1}^i 2^{k_j},i=1,\dots,l-1\right\}.
\]
which is exactly what we want.
\end{itemize}
%\end{itemize}

To complete the proof, it remains to show the base case of the induction. We will show the base cases $N=2$ and $N=3$.
For $N=2$, note that the $(N+1)$-cycle is simply a triangle which is already triangulated and the frequencies are simply $\{-1,0,1\}$. If we evaluate expression \eqref{eq:freqtriangulation} for $N=2$ (note that here $k=0$) we get $\cT = \{-1,0,1\}$, as needed.

For $N=3$ (the 4-cycle), we have $k=1$ and $l=2$ with $k_1 = 0$ and $k_2 = 1$. Thus expression \eqref{eq:freqtriangulation} evaluates to $\cT = \{0\} \cup \{\pm 1, \pm 2\} \cup \{1\} = \{-2,-1,0,1,2\}$. It is easy to construct a triangulation of the 4-cycle with such frequencies (one can even construct one where $\cT = \cK \cup (-\cK) = \{-1,0,1\}$).
\end{proof}

\begin{example}
Figure \ref{fig:recursive_triangulation_N=8} shows the recursive construction for the case $N=8$. We have indicated in each triangle (3-clique) the associated set of frequencies.
\begin{figure}[ht]
  \centering
  \includegraphics[scale=1]{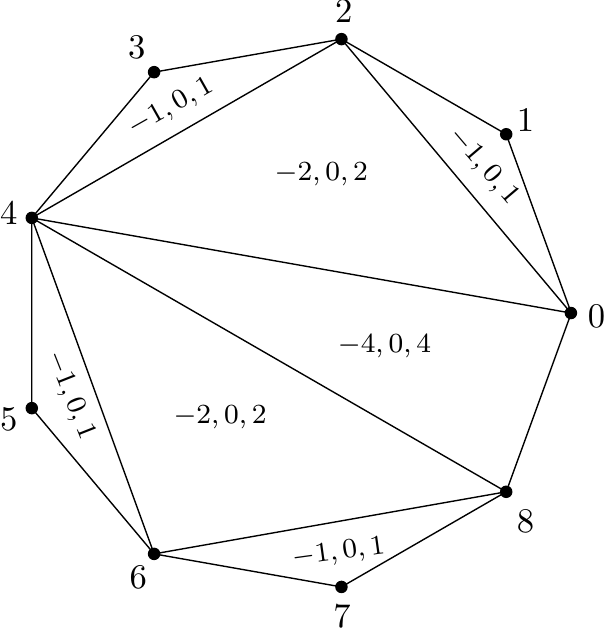}
  \caption{Illustration of the recursive triangulation of the $(N+1)$-cycle for $N=8$.}
  \label{fig:recursive_triangulation_N=8}
\end{figure}
\end{example}

\begin{proof}[Proof of Theorem \ref{thm:maincycle}]
To prove Theorem \ref{thm:maincycle} for the $N$-cycle, we use the triangulation of the $(N+1)$-cycle of Theorem \ref{thm:triangulationcycle} except that we regard nodes $0$ and $N$ as the same nodes (they collapse into a single one). Thus this means that the triangle in Figure \ref{fig:recursive_triangulation_log2} with frequencies $\{-2^k,0,N-2^k\}$ also collapses and we only have to look at the frequencies for parts (a) and (b). It is not hard to show that the frequencies we get are the same as those given in Equation \eqref{eq:freqtriangulation} except that in the middle term the iterate $i$ goes from $0$ to $k-1$ (instead of from $0$ to $k$), and in the last term the iterate $i$ goes from $1$ to $l-2$ (instead of from $1$ to $l-1$) which gives exactly the set of frequencies of Equation \eqref{eq:freqsos}.
\end{proof}

Note that there are actually many different ways of constructing triangulations for the cycle graph, and different constructions will lead to a different set of ``frequencies''. We can mention that for the cycle graph $C_{N}$ one can actually construct a triangulation where the number of frequencies is related to the logarithm of $N$ base 3. When $N$ is a power of three the frequencies are precisely the powers of 3 that are smaller than $N$. We omit the precise description of this construction, but Figure \ref{fig:triangulation_power3} shows the triangulation for the 9-cycle and 27-cycle.

\begin{figure}[ht]
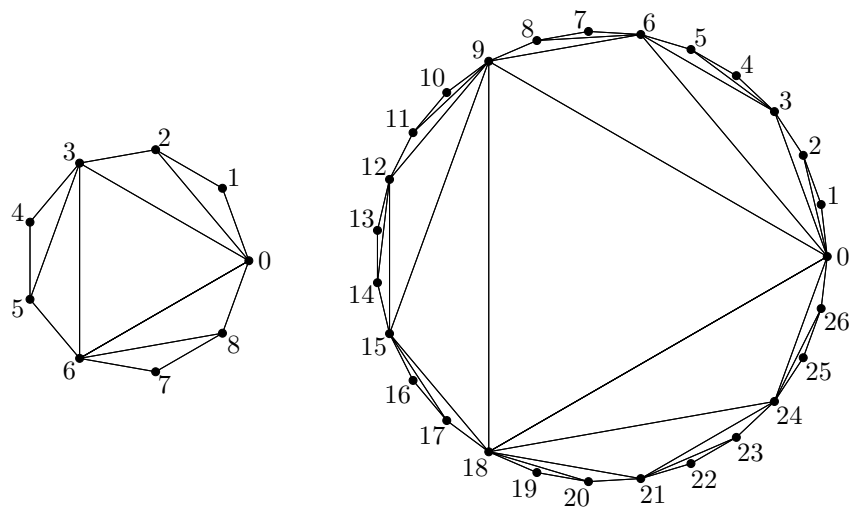

 \centering
 \begin{minipage}{\textwidth}
\centering
 \raisebox{-0.5\height}{\triangulationpowerofthree{9}{1}{1.5}}
 \qquad
  \raisebox{-0.5\height}{\triangulationpowerofthree{27}{2}{3}}
 \end{minipage}
 \caption{Triangulation of the 9-cycle with frequencies $\cT = \{0,\pm 1,\pm 3\}$ and of the 27-cycle with frequencies $\cT=\{0,\pm 1,\pm 3,\pm 9\}$.}
 \label{fig:triangulation_power3}
\end{figure}

\clearpage

\bibliographystyle{alpha}
\bibliography{../bib/psd_lifts}

\end{document}